\documentclass{amsart}
\usepackage{amsmath}
\usepackage{}
\usepackage{amsfonts}
\usepackage[colorlinks,citecolor=red,linkcolor=red]{hyperref}

\def \Vol{{\rm{Vol}}}
\def \diam{{\rm{diam}}}
\def \dist{{\rm{dist}}}
\def \Ric{{\rm{Ric}}}

\newtheorem{theorem}{Theorem}[section]
\newtheorem{lemma}[theorem]{Lemma}

\theoremstyle{definition}

\newtheorem{conjecture}[theorem]{Conjecture}

\newtheorem{corollary}[theorem]{Corollary}
\newtheorem{proposition}[theorem]{Proposition}
\theoremstyle{remark}
\newtheorem{remark}[theorem]{Remark}

\numberwithin{equation}{section}

\begin{document}

\title{Bergman kernel along the K\"{a}hler Ricci flow and Tian's Conjecture}

\author{Wenshuai Jiang}
    \address{BICMR, Peking University,Yiheyuan Road 5, Beijing, 100871, China}
    \email{jiangwenshuai@pku.edu.cn}

\keywords{K\"{a}hler Ricci flow, scalar curvature, Bergman kernel, partial $C^0$-estimate}

\begin{abstract}
In this paper, we study the behavior of Bergman kernels along the K\"{a}hler Ricci flow on Fano manifolds. We show that the Bergman kernels
 are equivalent along the K\"{a}hler Ricci flow under certain condition on the Ricci curvature of the initial metric. Then, using a recent work of Tian and Zhang,
 we can solve a conjecture of Tian for Fano manifolds of complex dimension $\le 3$.
\end{abstract}

\maketitle

\section{Introduction}
A Fano manifold is a compact K\"{a}hler manifold with positive first Chern
class. It has been one of main problems in K\"{a}hler geometry to study if a
Fano manifold admits a K\"{a}hler-Einstein metric since the Calabi-Yau
theorem on Ricci-flat K\"{a}hler metrics in 70's and the Aubin-Yau theorem
on K\"{a}hler-Einstein metrics with negative scalar curvature. This problem is more difficult
because there are new obstructions to the existence. It was
conjectured that the existence of K\"{a}hler-Einstein metrics on M is equivalent to
the K-stability, the Yau-Tian-Donaldson conjecture in the case of Fano manifolds.
\begin{theorem}
Let $M$ be a Fano manifold without non-zero holomorphic vector fields, then $M$ admits a K\"{a}hler-Einstein metric
if and only if it is K-stable.
\end{theorem}
The necessary part of this theorem is proved by Tian in \cite{T2}. Last Fall, Tian gave a
proof for the sufficient part (see \cite{T3}) by establishing the partial $C^0$-estimate for conic K\"{a}hler-
Einstein metrics. Another proof for the sufficient part was given in \cite{CDS1,CDS2,CDS3}

An older approach for solving the conjecture is to solve the following complex Monge-Amper\'{e} equations by the continuity method:
\begin{equation}\label{cma}
(\omega+\sqrt{-1}\partial\overline{\partial}\varphi)^m\,=\,e^{h-t\varphi}\,\omega^m,~~~~\omega + \sqrt{-1}\partial\overline{\partial}\varphi\,>\,0,
\end{equation}
where $\omega$ is a given K\"{a}hler metric with its K\"{a}hler class $[\omega] = 2\pi c_1(M)$ and $h$ is uniquely determined by
\begin{equation*}
\Ric(\omega)\,-\,\omega\,=\,\sqrt{-1}\partial\overline{\partial}h,~~~~\int_M(e^h-1)\omega^m\,=\,0.
\end{equation*}
Let $I$ be the set of t for which (\ref{cma}) is solvable. Then we have known: (1) By
the well-known Calabi-Yau theorem, $I$ is non-empty; (2) In 1983, Aubin proved
that $I$ is open; (3) If we can have an a prior $C^0$-estimate for the solutions
of (\ref{cma}), then $I$ is closed and consequently, there is a K\"{a}hler-Einstein metric on
$M$.

However, the $C^0$-estimate does not hold in general since there are many
Fano manifolds which do not admit any K\"{a}hler-Einstein metrics. The existence
of K\"{a}hler-Einstein metrics required certain geometric stability on the underlying
Fano manifolds. In 90's, Tian proposed a program towards establishing
the existence of K\"{a}hler-Einstein metrics. The key technical ingredient of this
program is a partial $C^0$-estimate conjecture( Conjecture \ref{conj}). Tian affirmed that if one can prove this conjecture( Conjecture \ref{conj})
for the solutions of (\ref{cma}), then one can use the K-stability to derive the a prior
$C^0$-estimate for the solutions of (\ref{cma}), and consequently, the existence of K\"{a}hler-Einstein metrics.
In this paper, we will solve Tian's partial $C^0$-estimate conjecture for complex dimension $\le 3$.

Let $(M,\omega)$ be a Fano manifold and $K_M^{-1}$ be its anti-canonical bundle. Choose a hermitian metric $H_{\omega}$ with $\omega$ as its curvature form
and any orthonormal basis $\{S_i\}_{1\le i\le N}$ of $H^0(M,K^{-l}_M)$,
with respect to the induced inner product induced by $H^{\otimes l}_{\omega}$ and $\omega$, where
$N=\dim H^0(M,K^{-l}_M)$. Then, following \cite{T3}, we define the Bergman kernel by
\begin{equation}
\rho_{\omega,l}(x)\,=\,\sum_{i=1}^N||S_i||_{H^{\otimes l}_{\omega}}^2(x).
\end{equation}
This is independent of the choice of $H_{\omega}$ and the orthonormal basis $\{S_i\}.$

\begin{remark}\label{rhoremark}
Denote
$$||S||^2_{{H^{\otimes l}_{\omega}},\omega}\,=\,\int_M||S||^2_{H^{\otimes l}_{\omega}}(x) d\mu(x)$$
and
$$\eta_\omega(x)\,=\,\sup_{||S||^2_{{H^{\otimes l}_{\omega}},\omega}=1}||S||^2_{H^{\otimes l}_{\omega}}(x),$$
then we have the following inequalities:
\begin{equation}
\frac{1}{N}\rho_{\omega,l}(x)\le \eta_\omega(x)\le \rho_{\omega,l}(x).
\end{equation}
\end{remark}

Denote by $\mathcal {K}(R_0,V_0,m)$ the set of all compact K\"ahler manifolds $(M,\omega)$ of complex dimension $m$ such that $[\omega] = 2\pi c_1(M)$,
$\Ric (\omega)\ge R_0>0$ and $\mbox{\Vol}(M,\omega)\ge V_0>0$. In 1990, Tian proposed the following conjecture:
\begin{conjecture} [Tian \cite{T4,Ti10}]\label{conj}
For each $(M,\omega)\in \mathcal{K}(R_0,V_0,m)$, there are uniform constants $c_k=c(m,k,R_0,V_0)>0$ ($k\ge 1$) and $l_i \to \infty$ with $i \ge 0$ and $l_0=l_0(m)$,
such that for all $l=l_i$, we have
\begin{equation*}
\rho_{\omega,l}\,>\,c_l\,>\,0.
\end{equation*}
\end{conjecture}
\begin{remark}
Tian also mentioned a stronger version of Conjecture \ref{conj}: There are uniform constants $c_k=c(m,k,R_0,V_0)>0$ for $k\ge 0$ and $l_0=l_0(m)$, such that
for all $l\ge l_0$, we have $\rho_{\omega,l}>c_l>0$.
\end{remark}
In this paper, using recent regularity theory developed by Tian and Zhang \cite{TZha1} for K\"{a}hler Ricci flow, we solve Tian's conjecture for complex dimensions $2$ and $3$:
\begin{theorem}\label{thm23}
Let $(M,\omega)$ be a compact K\"{a}hler manifold of complex dimension $2$ or $3$ and. We further assume a positive Ricci curvature lower bound
$\Ric(\omega)\ge R_0>0$ and a volume lower bound $V\ge V_0>0$.
Then there are uniform constants $c_k=c(m,k,R_0,V_0)>0$ for all $k\ge 1$ and $l_i \to \infty$ with $i \ge 0$ and $l_0=l_0(m, R_0, V_0)$,
such that for all $l=l_i$, we have
\begin{equation*}
\rho_{\omega,l}>c_l>0.
\end{equation*}
\end{theorem}
Following Tian's approach, we can prove(see also \cite{Zha})
\begin{corollary}
The Yau-Tian-Donaldson conjecture holds for complex dimension $\le 3.$
\end{corollary}

We can also prove a stronger version of Tian' Conjecture in complex dimension $1$,
\begin{theorem}\label{thm2}
Let $(M,\omega)$ be a K\"{a}hler manifold of complex dimension $1$ and with positive first Chern class. Then for all $l\in \mathbb{N}_{+}$,
the Bergman kernel $\rho_{\omega,l}$ has a uniform positive lower bound $c_{l}>0$, depending only on positive curvature lower bound, volume lower bound and $l$:
\begin{equation}
\rho_{\omega,l}\,>\,c_l\,>\,0.
\end{equation}
\end{theorem}
In order to study the Bergman kernel, we consider the K\"{a}hler Ricci flow
 \begin{equation}\label{Kahler}
\partial_t g_{i\bar{j}}\,=\,g_{i\bar{j}}\,-\,R_{i\bar{j}}=u_{i\bar{j}},~~t>0
\end{equation}
on a compact K\"{a}hler manifold $M$ of complex dimension $m$ and with $c_1(M) > 0$.
We can show that the Bergman kernels are equivalent along the K\"{a}hler Ricci flow (\ref{Kahler}), see Theorem \ref{Beq}.

Given any initial K\"{a}hler metric $g(0)$, Cao \cite{Ca} proved that (\ref{Kahler})
has a solution for all time $t \ge 0$. Moreover, Perelman (see \cite{ST}) proved that
the scalar curvature $R$ is uniformly bounded, and the Ricci potential $u$ is uniformly bounded in $C^1$ norm, with respect to $g(t)$.
One can easily check that these uniform bounds depend on the Sobolev constant $C_s$ of $g(0)$, the volume $V$ of $g(0)$,
the scalar curvature upper and lower bound of $g(0)$, the upper bound of $|\nabla u|(x,0)$. The following theorem is essentially due to Perelman,
\begin{theorem}\label{Perelman}(see \cite{ST})
Let $g(t)$ be a K\"{a}hler Ricci flow (\ref{Kahler}) on a Fano manifold $M$ of complex dimension $m$. There exists a uniform constant $C$ so that
\begin{equation*}
|R(g(t))|\,\le\, C,~ ~\mbox{\diam}(M,g(t))\,\le\, C,~~||u||_{C^1}\,\le\, C,
\end{equation*}
where the constant $C$ depends only on dimension $m$, \Vol$(M,g(0))$, the $L^2$-Sobolev constant $C_s$ of $g(0)$, bounds of $|R(g(0))|$ and $|\nabla u|(0)$.
\end{theorem}
In this paper, we prove that in all complex dimensions, the scalar curvature and the gradient of Ricci potential $|\nabla u|^2$ for $(M,g(t))$ have a
bound $C t^{-\frac{n+2}{2}}$ in small time. Here $C$ depends only on a lower bound of Ricci curvature, the volume bound of $g(0)$ and an upper bound of \diam $(M,g(0))$. The key result of this paper is the following:
\begin{theorem}\label{main}
Let $g(t)$ be a K\"{a}hler Ricci flow (\ref{Kahler}) on a Fano manifold $M$ of real dimension $n = 2m$.
Then there exists a uniform constant $C$, which depends only on constant $n_0$ ($n_0=n$ if $n\ge3$ , $n_0>2$ if $n=2$),
the lower bound of $\Ric(g(0))$, the volume $V$ of $g(0)$ and the upper bound of diameter of $g(0)$.
Such that for all $0<t<1$, we have
\begin{equation}
\sup_{x\in M}|R(x,t)|\,\le\, \frac{C}{t^{\frac{n_0+2}{2}}}
\end{equation}
and
\begin{equation}
\sup_{x\in M}|\nabla u(x,t)|^2\,\le\, \frac{C}{t^{\frac{n_0+2}{2}}}.
\end{equation}
\end{theorem}
\begin{remark} Note that, one can deduce a uniform lower bound for $u$ by Moser's iteration using the $L^2$-Sobolev inequality along the K\"{a}hler Ricci flow (Lemma \ref{Plem1}), so the gradient bound of $u$ implies the $C^1$-norm of $u$. Applying the results of Perelman (Theorem \ref{Perelman})and Theorem \ref{main}, we can estimate the diameter, the $C^1$ norm of Ricci potential, scalar curvature upper bound for all time along the
K\"{a}hler Ricci flow.
\end{remark}
Theorem \ref{main} is a corollary of the following theorem. Its proof replies on Moser's iteration combined with the $L^2$-Sobolev inequality along the Ricci flow
(see \cite{M}). Note that the term involving the scalar curvature will be a good term when applying Moser's iteration.
\begin{theorem}\label{thm1}
Let $g(t)$ be a K\"{a}hler Ricci flow  on a Fano manifold $M$ of real dimension $n$.
\begin{equation*}
\partial_t g_{i\bar{j}}\,=\,g_{i\bar{j}}\,-\,R_{i\bar{j}}\,=\,u_{i\bar{j}},~~t>0
\end{equation*}
Let $f$ be a nonnegative Lipschitz continous function on $M\times [0,\infty)$ satisfying
\begin{equation}\label{eqn1}
\frac{\partial f}{\partial t}\,\le\, \Delta f \,+\, a f
\end{equation}
on $M\times [0,\infty)$ in the weak sense, where $a\ge 0$,then for any $0<t<1$, $p>0$, we have
\begin{equation*}
\sup_{x\in M}|f(x,t)|\,\le\, \frac{C}{t^{\frac{n_0+2}{2p}}}\left(\int_0^1\int_M f(x,t)^pd\mu(t)dt\right)^\frac{1}{p},
\end{equation*}
where $C$ is a positive constant depending only on $a$, $p$, constant $n_0(n_0=n,if~n\ge3 ,~n_0>2,if~n=2)$, Sobolev constant $C_s$ of initial metric, volume $V$ of initial metric and negative lower bound of $R(0)$.
\end{theorem}

Notations: Let $R=R(x,t)$ be the scalar curvature at time $t$, $V$ be the volume of $g(0)$, $d$ be the diameter upper bound of $g(0)$, $d\mu(t)$ be the volume element of $g(t)$. Denote $\diam(M,g)$ and $\Vol(M,g)$ to be the diameter and volume of $(M,g)$.

The organization of this paper is as follows: In the next section, we give a proof for Theorem \ref{Perelman}, we will consider carefully how the quantities rely on the initial metric. In section \ref{secThm1}, we prove Theorem \ref{thm1}. The main idea is Moser's iteration for parabolic equation (see \cite{M}). In section \ref{secThm2}, applying Theorem \ref{thm1}, we give the proof of theorem \ref{main}. In section \ref{secPCE}, we consider the complex $1$ case and prove Theorem \ref{thm2}. We will divide the proof into several lemmas. In section \ref{secPCEH}, we consider high dimension case, we show that Bergman kernels are equivalent along the K\"{a}hler Ricci flow and complete the proof of Theorem \ref{thm23}.

\section{Perelman's Scalar Curvature Estimate}\label{secPere}
In this section, we will give a proof of Theorem \ref{Perelman}. The method is mainly similar to \cite{ST} and \cite{Ca1}. We will consider carefully all the quantities how to rely on the initial metric. We only prove complex dimension $\ge 2$ case. For
complex dimension $1$ Fano manifold, the proof is similar by noticing that Proposition \ref{Ppro} can be replaced with Lemma \ref{Sobolev2}.
First of all, we will show that the Ricci potential has a uniform lower bound, and then using maximum principle we can control the gradient of Ricci potential and scalar curvature upper bound. At last, a diameter upper bound estimate will conclude the proof.

Now we will prove a uniform Ricci potential lower bound. Firstly, we need to show the scalar curvature has a uniform lower bound.
\begin{lemma}
There exists a constant $C>0$ such that the scalar curvature $R$ of $g(t)$ satisfies the estimate
$$R(x,t)\,\ge\, -C$$
for all $t\ge 0$ and all $x\in M$. Here constant $C$ depends only on the lower bound of $R(g(0))$.
\end{lemma}
\begin{proof}
 By directly computing, we have the evolution of $R$,
 \begin{equation*}
 \frac{\partial}{\partial t}R\,=\, \Delta R\,+\,|\Ric|^2\,-\,R.
 \end{equation*}
 Let $R_{min}(0)$ be the minimum of $R(x,0)$ on $M$. If $R_{min}\ge 0$, then by maximum principle, we have $R(x,t)\ge 0$ for all $t>0$ and all $x\in M$.

 Now suppose $R_{min}(0)\,<0$. Set $F(x,t)\,=\, R(x,t)\,-\,R_{min}(0)$. Then, $F(x,0)\ge 0$ and $F$ satisfies
 $$\frac{\partial}{\partial t}F\,=\,\Delta F+|\Ric|^2-F-R_{min}(0)\,>\,\Delta +|\Ric|^2-F.$$
 Hence it follows again from the maximum principle that $F\ge 0$ on $M\times [0,\infty]$, i.e.,
 $$R(x,t)\,\ge\, R_{min}(0)$$
 for all $t>0$ and all $x\in M$.
\end{proof}
Next, we will show Perelman's $\kappa$ non-collapsing theorem, we need the following:
\begin{lemma}\label{Ple1}
Let $\hat{g}_{i\bar{j}}(s),~0\le s<1$ and $g_{i\bar{j}}(t),~0\le t<\infty$, be solutions to the K\"{a}hler Ricci flow (\ref{Peq1}) and (\ref{Peq2}) respectively,
 \begin{equation}\label{Peq1}
\partial_s \hat{g}_{i\bar{j}}(s)\,=\,-R_{i\bar{j}}(s),~~0\le s\,<\,1, ~\hat{g}_{i\bar{j}}(0)\,=\,g_{i\bar{j}}
\end{equation}
and
\begin{equation}\label{Peq2}
\partial_t g_{i\bar{j}}(t)\,=\,g_{i\bar{j}}(t)\,-\,R_{i\bar{j}}(t),~~t>0,~{g}_{i\bar{j}}(0)\,=\,g_{i\bar{j}}.
\end{equation}
Then $\hat{g}_{i\bar{j}}(s)$ and $g_{i\bar{j}}(t)$ are related by
\begin{equation*}
\hat{g}_{i\bar{j}}(s)\,=\,(1-s)g_{i\bar{j}}(t(s)),~~t\,=\,-\log(1-s)
\end{equation*}
and
\begin{equation*}
g_{i\bar{j}}(t)\,=\,e^t\hat{g}_{i\bar{j}}(s(t)),~~s\,=\,1-e^{-t}.
\end{equation*}
\qed
\end{lemma}
In order to show Perelman's $\kappa$-noncollapsed theorem, we only need to prove the following:
\begin{proposition}(Ye \cite{Y2})\label{Ppro}
Consider the K\"{a}hler Ricci flow (\ref{Peq2}) on Fano manifold. Then there are positive constants $A$ and $B$ depending only on the dimension $m$, a non-positive lower bound for $R_{g(0)}$, a positive lower bound for $\Vol(M,g(0))$, an upper bound for Sobolev constant $C_s(M,g(0))$. Such that, for each $t>0$ and all $f\in W^{1,2}(M)$ there holds
\begin{equation}
\left(\int_M f^{\frac{2m}{m-1}}d\mu(t)\right)^{\frac{m-1}{m}}\le A\int_M(|\nabla f|^2+\frac{R}{4}f^2)d\mu(t)\,+\,B\int_Mf^2d\mu(t).
\end{equation}
Consequently, let $L\,>\,0 $ and assume $R\le \frac{1}{r^2}$ on a geodesic ball $B(x,r)$ with $0\,<\,r\le L$. Then there holds
\begin{equation}
\Vol(B(x,r))\,\ge\, \left(\frac{1}{2^{2m+3}A\,+\,2BL^2}\right)^m r^{2m}.
\end{equation}
\end{proposition}
\begin{proof}
By the monotonicity of $\mathcal {W}$-entropy functional and the flow (\ref{Peq1}) in Lemma \ref{Ple1}, one can show that
\begin{equation*}
\left(\int_M f^{\frac{2m}{m-1}}d\hat{\mu}(s)\right)^{\frac{m-1}{m}}\le A\int_M(|{\nabla} f|^2_{\hat{g}}+\frac{\hat{R}}{4}f^2)d\hat{\mu}(s)\,+\,B\int_Mf^2d\hat{\mu}(s)
\end{equation*}
where constants $A$ and $B$ depend only on the quantities stated in the proposition. By the relation between (\ref{Peq1}) and (\ref{Peq2}),
we have
\begin{equation*}
\begin{split}
\left(\int_M f^{\frac{2m}{m-1}}d{\mu(t)}\right)^{\frac{m-1}{m}}&\le A\int_M(|\nabla f|^2+\frac{{R}}{4}f^2)d{\mu(t)}\,+\,Be^{-t}\int_Mf^2d{\mu(t)}\\
&\le A\int_M(|\nabla f|^2+\frac{{R}}{4}f^2)d{\mu(t)}\,+\,B\int_Mf^2d{\mu(t)}.
\end{split}
\end{equation*}
At last, to prove the $\kappa$-noncollapsed, we can assume $r=1$ since we can scale the metric with factor $\frac{1}{r^2}$. That is $\bar{g}=\frac{1}{r^2} g$. Thus we have
\begin{equation*}
\begin{split}
\left(\int_M f^{\frac{2m}{m-1}}d\bar{\mu}(t)\right)^{\frac{m-1}{m}}&\le A\int_M(|\nabla f|^2_{\bar{g}}+\frac{\bar{R}}{4}f^2)d\bar{\mu}(t)\,+\,Br^2\int_Mf^2d\bar{\mu}(t)\\
&\le  A\int_M(|\nabla f|^2_{\bar{g}}+\frac{\bar{R}}{4}f^2)d\bar{\mu}(t)\,+\,BL^2\int_Mf^2d\bar{\mu}(t)
\end{split}
\end{equation*}
and $\bar{R}\le 1$ on geodesic ball $B_{\bar{g}}(x,1)$. Then a standard argument implies the lower bound of $\Vol _{\bar{g}}(B_{\bar{g}}(x,1))$. One can find more details in Ye \cite{Y2}.
\end{proof}
By taking the trace of (\ref{Kahler}), we get $\Delta u\,=\,m\,-\,R.$ Normalize $u$ so that
\begin{equation*}
\int_Me^{-u}d\mu(t)\,=\,(2\pi)^m.
\end{equation*}
Then we have
\begin{lemma}\label{Plem1}
Function $u(t)$ is uniformly bounded from below. That is
\begin{equation*}
u(t)\ge -C,
\end{equation*}
where constant $C$ depends only on the constants in Proposition \ref{Ppro}.
\end{lemma}
\begin{proof}
Since $\Delta u\,=\,m\,-\,R,$ we have
\begin{equation*}
\begin{split}
\Delta e^{-u}\,=\,-\Delta ue^{-u}+|\nabla u|^2 e^{-u}\,\ge\, (R-m) e^{-u}.
\end{split}
\end{equation*}
Denote $f\,=\,e^{-u}$. We have
\begin{equation}\label{Peq3}
-\Delta f+R f\,\le\, m f.
\end{equation}
Multiplying $f^p$ to both sides of (\ref{Peq3}) and integrating by parts, we have
\begin{equation*}
\frac{4p}{(p+1)^2}\int_M|\nabla f^{\frac{p+1}{2}}|^2d\mu(t)\,+\,\int_MR\,f^{p+1}d\mu(t)\,\le\, m\int_Mf^{p+1}d\mu(t).
\end{equation*}
That is
\begin{equation*}
\int_M|\nabla f^{\frac{p+1}{2}}|^2d\mu(t)\,+\,\frac{(p+1)^2}{p}\int_M\frac{R}{4}\,f^{p+1}d\mu(t)\,\le\, \frac{m(p+1)^2}{4p}\int_Mf^{p+1}d\mu(t).
\end{equation*}
Since $R$ has a uniform lower bound, so
\begin{equation*}
\int_M|\nabla f^{\frac{p+1}{2}}|^2d\mu(t)\,+\,\int_M\frac{R}{4}\,f^{p+1}d\mu(t)\,\le\, C\,p\int_Mf^{p+1}d\mu(t).
\end{equation*}
Then by Proposition \ref{Ppro} and Moser's iteration, we deduce
\begin{equation*}
\sup_{x\in M}f(x)\,\le\, C\int_Mfd\mu(t)\,=\,C\int_Me^{-u}d\mu(t)\,=\,C(2\pi)^m.
\end{equation*}
This provides a pointwise lower bound of $u$.
\end{proof}
Define
\begin{equation*}
\mathcal {W}(g,f,\tau)\,=\,(4\pi \tau)^{-m}\int_Me^{-f}\{2\tau(R+|\nabla f|^2)+f-2m\}d\mu,~~\int_Me^{-f}d\mu\,=\,(4\pi \tau)^m
\end{equation*}
to be  Perelman's $\mathcal{W}$-entropy functional for $g$ as in \cite{Pe}. By directly computing, we have
\begin{equation*}
\frac{\partial}{\partial t}\mathcal{W}(g(t),f(t),\frac{1}{2})\,=\,(2\pi)^{-m} \int_Me^{-f}\left(|\Ric+\nabla \overline{\nabla} f-\omega|^2+|\nabla \nabla f|^2\right)d\mu(t)\ge 0
\end{equation*}
along the following
\begin{equation*}
\begin{split}
&\partial_t f\,=\, -\Delta f\,+\, |\nabla f|^2\,-\, R\,+\,m,~~~\int_Me^{-f}d\mu(t)\,=\,(2\pi)^m,\\
&\partial_t g_{i\bar{j}}(t)\,=\,g_{i\bar{j}}(t)\,-\,R_{i\bar{j}}(t).
\end{split}
\end{equation*}
Define
\begin{equation*}
\mu(g,\tau)\,=\,\inf_{\{f|\int_Me^{-f}d\mu\,=\,(4\pi \tau)^m\}}\mathcal{W}(g,f,\tau).
\end{equation*}
Then
\begin{equation*}
\begin{split}
A_0\,&=\,\mu(g(0),\frac{1}{2})\,\le\, \mu(g(t),\frac{1}{2})\\
&\le\, \int_M(2\pi)^{-m}e^{-u}(R+|\nabla u|^2+u-2n)d\mu(t)\\
&=\,\int_M(2\pi)^{-m}e^{-u}(-\Delta u+|\nabla u|^2+u-n)d\mu(t)\\
&=\,-m+(2\pi)^{-m}\int_Me^{-u}ud\mu(t).
\end{split}
\end{equation*}
On the other hand, let $F=e^{-\frac{f}{2}}(2\pi)^{-\frac{m}{2}}.$ Then $\int_MF^2d\mu(t)=1$.
\begin{equation*}
\begin{split}
\mathcal {W}(g,f,\frac{1}{2})\,&=\,\int_M \left(RF^2+4|\nabla F|^2-F^2\log F^2\right)d\mu-2m-m\log (2\pi)\ge -C.
\end{split}
\end{equation*}
Here we have used the $L^2$-Sobolev inequality along the K\"{a}hler Ricci flow and the uniform lower bound of scalar curvature. The constant depends only on the constants in Proposition \ref{Ppro}. Particularly, we have $A_0\ge -C$. Moreover, the function $xe^{-x}$ is bounded from above. Thus we have
\begin{lemma}
Denote $a\,=\,-(2\pi)^{-m}\int_Me^{-u}ud\mu(t)$, then
\begin{equation*}
|a(t)|\,\le\, C
\end{equation*}
where the constant $C$ depends only on the volume of $g(0)$, a lower bound of $R(g(0))$ and an upper bound of the Sobolev constant $C_s=C_s(M,g(0))$.
\qed
\end{lemma}
The Ricci potential $u(x,t)$ satisfies
\begin{equation*}
\partial_i\partial_{\bar{j}}u\,=\,g_{i\bar{j}}\,-\,R_{i\bar{j}}.
\end{equation*}
Differentiating this, we have
\begin{equation*}
\begin{split}
\partial _i\partial_{\bar{j}}u_t\,&=\,g_{i\bar{j}}\,-\,R_{i\bar{j}}\,+\, \frac{\partial}{\partial t}\partial _i\partial_{\bar{j}}\log {\rm{det}}(g_{i\bar{j}})\\
&=\partial _i\partial_{\bar{j}}(u+\Delta u),
\end{split}
\end{equation*}
which implies
\begin{equation}\label{Peqa}
\frac{\partial}{\partial t}u\,=\,\Delta u+u+\varphi (t).
\end{equation}
However,
\begin{equation*}
0\,=\,\frac{\partial}{\partial t}\int_Me^{-u}d\mu(t)\,=\,\int_Me^{-u}(-\partial_t u+\Delta u)d\mu(t)\,=\,\int_Me^{-u}(-u-\varphi(t))d\mu(t).
\end{equation*}
Thus
\begin{equation*}
\varphi(t)\,=\,-(2\pi)^{-m}\int_Me^{-u}ud\mu(t)\,=\,a.
\end{equation*}
By maximum principle, one can easily prove the following:
\begin{lemma}\label{Plem2}
There is a uniform constant $C$, so that
\begin{equation}\label{Pleq1}
|\nabla u|^2(x,t)\,\le\, C(u+C),
\end{equation}
\begin{equation}\label{Pleq2}
R\,\le\, C(u+C),
\end{equation}
where the constant depends only on \Vol$(M,g(0))$, the $L^2$-Sobolev constant $C_s$ of $g(0)$ and upper bounds of $|R(g(0))|$ and $|\nabla u|(0)$.
\end{lemma}
\begin{proof}
This is essentially a parabolic version of Yau's gradient estimate in \cite{SY}. By Lemma \ref{Plem1}, we have $u(x,t)\ge -C$. Choosing $B=C+1$, then $u(x,t)+B\ge 1$. Let $H=\frac{|\nabla u|^2}{u+B}$. In order to show (\ref{Pleq1}), we only need to estimate an upper bound for $H$. By directly computing, we have
\begin{equation}\label{PeqH}
\begin{split}
(\partial_t-\Delta) H&\,=\,\frac{(B-a)|\nabla u|^2}{(u+B)^2}-\frac{|\nabla \overline{\nabla} u|^2+|\nabla \nabla u|^2}{u+B}+\frac{2\langle \nabla |\nabla u|^2,\nabla u\rangle}{(u+B)^2}-\frac{2|\nabla u|^4}{(u+B)^3}\\
&\,=\,\frac{(B-a)|\nabla u|^2}{(u+B)^2}-\frac{|\nabla \overline{\nabla} u|^2+|\nabla \nabla u|^2}{u+B}+\frac{2\langle \nabla H,\nabla u\rangle}{u+B}.
\end{split}
\end{equation}
For each $T>0$, suppose $H$ attains its maximum at $(x_0,t_0)$ on $M\times [0,T]$. If $t_0=0$, the upper bound of $H$ follows easily by the bound for $|\nabla u|(g(0))$. Assume $t_0>0$. Then at $(x_0,t_0)$ we have
\begin{equation*}
\partial_tH(x_0,t_0)\,\ge\, 0,~~\nabla H(x_0,t_0)\,=\,0,~~\Delta H(x_0,t_0)\,\le\, 0.
\end{equation*}
Substituting these into (\ref{PeqH}), we obtain
\begin{equation}\label{Peq4}
\frac{(B-a)|\nabla u|^2}{u+B}(x_0,t_0)\,\ge\, {|\nabla \overline{\nabla} u|^2(x_0,t_0)\,+\,|\nabla \nabla u|^2}(x_0,t_0).
\end{equation}
On the other hand, since $\nabla H(x_0,t_0)=0$, we have
\begin{equation*}
\nabla |\nabla u|^2(x_0,t_0)\,=\,\frac{|\nabla u|^2\nabla u}{u+B}(x_0,t_0).
\end{equation*}
Thus at $(x_0,t_0)$,
\begin{equation*}
\begin{split}
\frac{|\nabla u|^3}{u+B}&\,=\,|\nabla |\nabla u|^2|\,\le\, |\nabla u|(|\nabla \overline{\nabla} u|+|\nabla \nabla u|)\\
&\,\le\, \sqrt{2}|\nabla u|(|\nabla \overline{\nabla} u|^2+|\nabla \nabla u|^2)^{\frac{1}{2}}.
\end{split}
\end{equation*}
Combining with (\ref{Peq4}), we have
\begin{equation*}
2(B-a)H(x_0,t_0)\,\ge\, H^2(x_0,t_0).
\end{equation*}
Hence $H(x_0,t_0)\le 2(B-a)$. Let $T\to \infty$ and note that $a$ is bounded. We complete the proof of (\ref{Pleq1}).

Now we turn to the proof of (\ref{Pleq2}). Our goal is to prove that $-\Delta u$ is bounded by $C(u+C)$, which yields (\ref{Pleq2}),
since $\Delta u=n-R$. Let $K=\frac{-\Delta u}{u+B}$, where $B$ is a uniform constant as above. Similar computation as before gives that
\begin{equation*}
(\partial_t-\Delta)K \,=\,\frac{|\nabla \overline{\nabla}u|^2}{u+B}+\frac{(-\Delta u)(B-a)}{(u+B)^2}+2\frac{\langle \nabla K,\nabla u\rangle}{u+B}.
\end{equation*}
Combining this with (\ref{PeqH}), we have
\begin{equation*}
(\partial_t-\Delta)(K+2H) =\frac{-|\nabla \overline{\nabla}u|^2-2|\nabla\nabla u|^2}{u+B}+\frac{(-\Delta u+2|\nabla u|^2)(B-a)}{(u+B)^2}+2\frac{\langle \nabla (K+2H),\nabla u\rangle}{u+B}.
\end{equation*}
For each $T>0$, suppose $2H+K$ attains its maximum at $(x_0,t_0)$ on $M\times [0,T]$. If $t_0=0$, the upper bound of $2H+K$ follows easily by the bound for $|\nabla u|(0)$ and $|R|(0)$. Assume $t_0>0$. Then at $(x_0,t_0)$ we have
\begin{equation*}
\frac{(-\Delta u+2|\nabla u|^2)(B-a)}{(u+B)^2}\,\ge\, \frac{|\nabla \overline{\nabla}u|^2+2|\nabla\nabla u|^2}{u+B}\,\ge\,\frac{|\nabla \overline{\nabla}u|^2}{u+B}\,\ge\, \frac{(\Delta u)^2}{m(u+B)}.
\end{equation*}
Thus
\begin{equation*}
\frac{1}{m}(\frac{\Delta u}{u+B})^2+(B-a)\frac{\Delta u}{u+B}\,\le\, \frac{2(B-a)|\nabla u|^2)}{u+B}\,\le\, C.
\end{equation*}
Here we have used the fact that $u+B\ge 1$. Hence $|\frac{\Delta u}{u+B}|(x_0,t_0)\le C$. Now for each $(x,t)\in M\times [0,T]$,
\begin{equation*}
\begin{split}
\frac{-\Delta u}{u+B}(x,t)&\,\le\, \frac{-\Delta u+2|\nabla u|^2}{u+B}(x,t)\\
&\,\le\, \frac{-\Delta u+2|\nabla u|^2}{u+B}(x_0,t_0)\\
&\,=\,\frac{-\Delta u}{u+B}(x_0,t_0)+ \frac{2|\nabla u|^2}{u+B}(x_0,t_0)\,\le\, C,
\end{split}
\end{equation*}
Let $T\to \infty $. We finish the proof.
\end{proof}
\begin{corollary}\label{Pcor}
There exits a constant $C$ depending only on the constant of Lemma \ref{Plem2}, such that,
\begin{equation*}
\begin{split}
&u(y,t)\,\le\, C(\dist_t^2(\hat{x},y)+1),\\
&R(y,t)\,\le\, C(\dist_t^2(\hat{x},y)+1),\\
&|\nabla u|^2(y,t))\,\le\, C(\dist_t^2(\hat{x},y)+1)
\end{split}
\end{equation*}
for all $t>0$ and $y\in M$, where $u(\hat{x},t)=\min_{x\in M}u(x,t).$
\end{corollary}
\begin{proof}
By Lemma \ref{Plem2}, we only need to estimate $u(x,t)$. Actually, by (\ref{Pleq1}), we have
\begin{equation*}
|\nabla \sqrt{u+C}|\le C.
\end{equation*}
Hence we have
\begin{equation*}
\sqrt{u+C}(x,t)\le \sqrt{u+C}(\hat{x},t)+C\dist_t(\hat{x},x),
\end{equation*}
where $u(\hat{x},t)=\min_{x\in M}u(x,t)$. On the other hand, since $\int_Me^{-u}d\mu(t)=(2\pi)^m$, we have
\begin{equation*}
u(\hat{x},t)\le \log (\frac{V}{(2\pi)^m}).
\end{equation*}
So
\begin{equation*}
u(y,t)\le C(\dist_t^2(\hat{x},y)+1).
\end{equation*}
The other two inequalities follow from this and Lemma \ref{Plem2}.
\end{proof}
Notice the results in Corollary \ref{Pcor}. To prove Theorem \ref{Perelman}, it suffices to estimate the diameter upper bound.
Let $B(k_1,k_2)=\{z: 2^{k_1}\le \dist_t(\hat{x},z)\le 2^{k_2}\}$. Consider an annular $B(k,k+1)$. By Corollary \ref{Pcor} we have that $R\le C 2^{2k}$ on $B(k,k+1)$ and note that $B(k,k+1)$ contains at least $2^{2k-1}$ balls of radii $\frac{1}{2^k}.$ By Proposition \ref{Ppro} we have
\begin{equation}\label{Peq5}
\Vol(B(k,k+1))\ge \sum_i \Vol(B(x_i,2^{-k}))\ge C2^{2k-2km},
\end{equation}
where the constant $C$ depends only on the constant in Corollary \ref{Pcor} and constants in Proposition \ref{Ppro}.
\begin{lemma}\label{Plem4}
For each $\epsilon>0$, if $\diam(M,g(t))\ge C_{\epsilon}$, we can find $B(k_1,k_2)$ such that
\begin{equation*}
\begin{split}
\Vol(B(k_1,k_2))&\,<\,\epsilon,\\
\Vol(B(k_1,k_2))&\,\le\, 2^{10m}\Vol(B(k_1+2,k_2-2)).
\end{split}
\end{equation*}
Here we can choose $C_{\epsilon}=2^{\left(\frac{\log(V/C)}{(2m+8)\log 2}+2\right)4^{(\frac{V}{\epsilon}+2)}+1}$ and $C$ is the constant in (\ref{Peq5})
\end{lemma}
\begin{proof}
Denote  $k_0=\frac{\log(V/C)}{(2m+8)\log 2}+2$ and assume $\diam(M,g)\ge 2^{k_04^{[\frac{V}{\epsilon}]+1}+1}$. Then we will show that for each $ \frac{k_0}{2} \le k\le \frac{k_0}{2}4^{[\frac{V}{\epsilon}]}$, there exits $B(k_1,k_2)$ so that $2k\le k_1<k_2\le 6k+1$ and
\begin{equation*}
\Vol(B(k_1,k_2))\le 2^{10m}\Vol(B(k_1+2,k_2-2)).
\end{equation*}
Otherwise, by (\ref{Peq5})
\begin{equation*}
\begin{split}
V\ge \Vol(B(2k,6k+1))&>2^{10m}\Vol(B(2k+2,6k-1))\\
&> 2^{10mk}\Vol(B(4k,4k+1))\\
&\ge C2^{10mk}2^{8k-8km}=C2^{2km+8k}.
\end{split}
\end{equation*}
Thus
\begin{equation*}
k\le \frac{\log(V/C)}{2(2m+8)\log 2}.
\end{equation*}

On the other hand, there must be some $0\le l\le [\frac{V}{\epsilon}]$ such that
\begin{equation*}
\Vol(B(k_04^l,k_04^{l+1}))<\epsilon.
\end{equation*}
Otherwise,
\begin{equation*}
V\ge \sum_{l=0}^{[\frac{V}{\epsilon}]}\Vol(B(k_04^l,k_04^{l+1}))\ge ([\frac{V}{\epsilon}]+1)\epsilon >V.
\end{equation*}
Getting together all the above arguments will imply the lemma.
\end{proof}

\begin{lemma}\label{Plem3}
For each $0<k_1<k_2<\infty$, there exist $r_1,\,r_2$ and a uniform constant $C$ such that $2^{k_1}\le r_1\le 2^{k_1+1},~2^{k_2-1}\le r_2\le 2^{k_1}$ and
\begin{equation*}
\int_{B(r_1,r_2)}Rd\mu(t)\le C~\Vol(B(k_1,k_2)),
\end{equation*}
where $B(r_1,r_2)=\{z\in M: r_1\le \dist_t(z,\hat{x})\le r_2\}$ and the constant $C$ depends only on the constant in Corollary \ref{Pcor}.
\end{lemma}
\begin{proof}
First of all, since
\begin{equation*}
\frac{d}{dr}\Vol(B(r))=\Vol(S(r)),
\end{equation*}
we have
\begin{equation*}
\Vol(B(k_1,k_1+1))=\int_{2^{k_1}}^{2^{k_1+1}}\Vol(S(r))dr.
\end{equation*}
Here $S(r)$ denotes the geodesic sphere of radius $r$ centered at $\hat{x}$ with respect to $g(t)$.

Hence, we can choose $r_1\in [2^{k_1},2^{k_1+1}]$ such that
\begin{equation*}
\Vol(S(r_1))\le \frac{\Vol(B(k_1,k_1+1))}{2^{k_1}}\le \frac{\Vol(B(k_1,k_2))}{2^{k_1}}.
\end{equation*}
Similarly, there exists $r_2\in [2^{k_2-1},2^{k_2}]$ such that
\begin{equation*}
\Vol(S(r_2))\le \frac{\Vol(B(k_2-1,k_2))}{2^{k_1}}\le \frac{\Vol(B(k_1,k_2))}{2^{k_2}}.
\end{equation*}
Next, by integration by parts and Corollary \ref{Pcor},
\begin{equation*}
\begin{split}
|\int_{B(r_1,r_2)}\Delta ud\mu(t)|&\le \int_{S(r_1)}|\nabla u|d\sigma(t)+\int_{S(r_2)}|\nabla u|d\sigma(t)\\
&\le \frac{\Vol(B(k_1,k_2))}{2^{k_1}}C2^{k_1+1}+\frac{\Vol(B(k_1,k_2))}{2^{k_2}}C2^{k_2+1}\\
&\le 4C\Vol(B(k_1,k_2)).
\end{split}
\end{equation*}
Therefore, since $R=-\Delta u+m$, it follows that
\begin{equation*}
\int_{B(r_1,r_2)}Rd\mu(t)\le (m+4)C~\Vol(B(k_1,k_2))
\end{equation*}
proving Lemma \ref{Plem3}. Here $C$ is the constant in Corollary \ref{Pcor}.
\end{proof}
In order to control the diameter of $M$, we only need to show the following:
\begin{lemma}\label{Plem5}
There exists a constant $\epsilon_0>0$. If $0<\epsilon<\epsilon_0$, we can't find $B(k_1,k_2)$ such that
\begin{equation}\label{Peq6}
\begin{split}
\Vol(B(k_1,k_2))&\,<\,\epsilon,\\
\Vol(B(k_1,k_2))&\,\le\, 2^{10m}\Vol(B(k_1+2,k_2-2)).
\end{split}
\end{equation}
Here we can choose
$\epsilon_0=(2\pi)^m e^{6C2^{10m}+A_0+2m}$ and constant $C$ is the constant in Lemma \ref{Plem3} and $A_0=\mu(g(0),\frac{1}{2}).$
\end{lemma}
\begin{proof}
Actually, if we can find $B(k_1,k_2)$ such that (\ref{Peq6}) holds. Then we choose $r_1,~r_2$ as in Lemma \ref{Plem3}. Define a cut off function $0\le\phi\le 1$,
 \[\phi(s) = \left\{\begin{array}{l}
 1, ~~2^{k_1+2}\le s\le 2^{k_2-2}, \\
 0, ~~\mbox{outside~}[r_1,r_2].\\
 \end{array} \right.\]
Then $|\phi'|\le 1$ everywhere. Let
$$F(y)\,=\,e^L\phi(\dist_t(\hat{x},y)),$$
where the constant $L$ is chosen so that
$$(2\pi)^m=\int_MF^2d\mu(t)=e^{2L}\int_{B(r_1,r_2)}\phi^2d\mu(t).$$
Since $\Vol(B(r_1,r_2))\le \Vol(B(k_1,k_2))<\epsilon$, thus $L\ge \frac{1}{2}\log \left(\frac{(2\pi)^m}{\epsilon}\right).$

By monotonicity of $\mathcal{W}$-entropy functional, we have
\begin{equation*}
\begin{split}
A_0\,&=\,\mu(g(0),\frac{1}{2})\le \mu(g(t),\frac{1}{2})\\
&\le\, (2\pi)^{-m}\int_M \left(RF^2+4|\nabla F|^2-F^2\log F^2\right)d\mu(t)-2m\\
&=\,(2\pi)^{-m}e^{2L}\int_{B(r_1,r_2)}\left(R\phi^2+4|\phi'|^2-\phi^2\log \phi^2\right)d\mu(t)-2L-2m.
\end{split}
\end{equation*}
By Lemma \ref{Plem3} we have
\begin{equation*}
\begin{split}
e^{2L}\int_{B(r_1,r_2)}R\phi^2d\mu(t)&\le Ce^{2L}\Vol(B(k_1,k_2))\\
&\le C e^{2L}2^{10m} \Vol(B(k_1+2,k_2-2))\\
&\le C2^{10m}\int_{B(r_1,r_2)}F^2d\mu(t)=C2^{10m}(2\pi)^m.
\end{split}
\end{equation*}
On the other hand, using $|\phi'|\le 1$ and $-s\log s\le e^{-1}$,  we have
\begin{equation*}
\begin{split}
e^{2L}\int_{B(r_1,r_2)}\left(4|\phi'|^2-\phi^2\log \phi^2\right)d\mu(t)&\le 5 e^{2L}\Vol(B(k_1,k_2))\\
&\le 5 Ce^{2L}2^{10m} \Vol(B(k_1+2,k_2-2))\\
&\le 5C2^{10m}\int_{B(r_1,r_2)}F^2d\mu(t)=5C2^{10m}(2\pi)^m
\end{split}
\end{equation*}
for $0\le s\le 1$. The above constant $C$ is the uniform constant in Lemma \ref{Plem3}. Therefore,
\begin{equation*}
A_0\le -2(L+m)+6C2^{10m}.
\end{equation*}
Hence we have
\begin{equation*}
\log \left(\frac{(2\pi)^m}{\epsilon}\right)\le 2L\le 6C2^{10m}-A_0-2m.
\end{equation*}
Thus it provides
\begin{equation*}
\epsilon\ge (2\pi)^m e^{6C2^{10m}+A_0+2m}
\end{equation*}
\end{proof}
Combining Lemma \ref{Plem4} and Lemma \ref{Plem5} will finish the proof of Theorem \ref{Perelman}.\qed

\section{A Linear Parabolic Estimate}\label{secThm1}
The main purpose of this section is to prove Theorem \ref{thm1}. We need two lemmas. The following lemma is due to Rugang Ye \cite{Y2} and Qi S. Zhang \cite{Z}, for $n\ge 3$,  see also Proposition \ref{Ppro},
\begin{lemma}\label{Sobolev1}
Let $(M,g(t))$ be a K\"{a}hler Ricci flow with real dimension $n$, $C_1(M)>0$. At time $t=0$, the following $L^2$ Sobolev inequality holds
\begin{equation*}
\left(\int_{M}f(x)^{\frac{2n}{n-2}}d\mu(0)\right)^{\frac{n-2}{n}}\le\, C_s\left(\int_M |\nabla f(x)|^2d\mu(0)+\int_M f^2(x)d\mu(0)\right).
\end{equation*}
Then along the K\"{a}hler Ricci flow we have
\begin{equation*}
\left(\int_{M}f(x)^{\frac{2n}{n-2}}d\mu(t)\right)^{\frac{n-2}{n}}\le\, A\left(\int_M 4|\nabla f(x)|^2+(R(x,t)+C_0)f^2(x)d\mu(t)\right)
\end{equation*}
for all $0\le t\le 1$, where the constants $A,~C_0$ depend only on dimension $n$, Sobolev constant $C_s$, volume $V$ of $g(0)$ and a lower bound of $R(g(0))$.
\end{lemma}
For $n=2$, the following lemma follows from Theorem 1 and Theorem 2 of \cite{Hs}, Lemma 4.1 of \cite{Y2} and a standard argument of \cite{BCLS}.
\begin{lemma}\label{Sobolev2}
Let $(M,g(t))$ be a K\"{a}hler Ricci flow with real dimension $2$, $C_1(M)>0$. At time $t=0$, the following $L^1$ Sobolev inequality holds
\begin{equation*}
\left(\int_{M}f(x)^{2}d\mu(0)\right)^{\frac{1}{2}}\,\le\, C_s\left(\int_M |\nabla f(x)|d\mu(0)+\int_M |f(x)|d\mu(0)\right).
\end{equation*}
Then along the K\"{a}hler Ricci flow holds
\begin{equation}\label{eqn7}
\left(\int_{M}f(x)^{\frac{2n_0}{n_0-2}}d\mu(t)\right)^{\frac{n_0-2}{n_0}}\,\le\,  A\left(\int_M 4|\nabla f(x)|^2+(R(x,t)+C_0)f^2(x)d\mu(t)\right)
\end{equation}
for all $0\le t\le 1$ and $n_0>2$, where the constants $A,~C_0$ depend only on constant $n_0$, Sobolev constant $C_s$, volume $V$ of $g(0)$ and a lower bound of $R(g(0))$ .
\end{lemma}
\noindent \textit{Outline of the proof.} Consider the flow (\ref{Peq1}). By Theorem $1$ and Theorem $2$ of \cite{Hs}, for all $0<s\le 1,~\sigma >0$, $f\in C^{\infty}(M,\hat{g}(s))$ and $||f||_{L^2(M,\hat{g}(s))}=1$, we obtain
\begin{equation}\label{LSoeq}
\int_Mf^2\log f^2d\hat{\mu}(s)\le \sigma \int_M\left(4|\nabla f(x)|^2_{\hat{g}}+\hat{R} f^2(x)\right)d\hat{\mu}(s)-\log \sigma+C.
\end{equation}
Taking the minimum of the right hand side respect to $\sigma$, then for all $n_0>2$, we have
\begin{equation*}
\begin{split}
\int_Mf^2\log f^2d\hat{\mu}(s)&\le  \log\left(C\int_M\left(4|\nabla f(x)|^2_{\hat{g}}+(\hat{R} +C_0)f^2(x)\right)d\hat{\mu}(s)\right) \\
&\le \frac{n_0}{2} \log\left(C\int_M\left(4|\nabla f(x)|^2_{\hat{g}}+(\hat{R} +C_0)f^2(x)\right)d\hat{\mu}(s)\right).
\end{split}
\end{equation*}
The following argument is standard, see section 10.2 of \cite{BCLS} and Proposition \ref{Ppro}.\qed
\begin{remark}
One can easily check that the constants $A$ in Lemma \ref{Sobolev1} and \ref{Sobolev2} do not depend on the time $t$, since the constant $C$ in (\ref{LSoeq}) doesn't depend on $t$.  However, here we only need $t\le 1$.
\end{remark}
Now we are ready to prove theorem \ref{thm1}:

 \noindent\textit{Proof of Theorem \ref{thm1}.}\\
 Case $n\ge3$: Let $f^p,~p\ge 1$ be a test function in (\ref{eqn1}). We have
 $$\int_M f^p f_t d\mu(t)-\int_Mf^p\Delta f d\mu(t)\,\le\, a \int_M f^{p+1}d\mu(t).$$
 Integration by parts, we have
 $$\frac{1}{p+1}\int_M(f^{p+1})_td\mu(t)+\frac{4p}{(p+1)^2}\int_M |\nabla f^{\frac{p+1}{2}}|^2d\mu(t)\le a \int_M f^{p+1}d\mu(t).$$
 Moreover, since $\partial_td\mu(t)=\Delta ud\mu(t)=(\frac{n}{2}-R)d\mu(t)$,
\begin{equation*}
\begin{split}\frac{1}{p+1}\partial_t\int_Mf^{p+1}d\mu(t)+\frac{1}{p+1}\int_MRf^{p+1}d\mu(t)&+\frac{4p}{(p+1)^2}\int_M |\nabla f^{\frac{p+1}{2}}|^2d\mu(t)\\
&\le (a+\frac{n}{p+1}) \int_M f^{p+1}d\mu(t).
\end{split}
\end{equation*}
  Multiplying both sides $p+1$, since $4p\ge 2(p+1)$ for all $p\ge 1$, we get
 $$\partial_t\int_Mf^{p+1}d\mu(t)+\int_MRf^{p+1}d\mu(t)+2\int_M |\nabla f^{\frac{p+1}{2}}|^2d\mu(t)\le (a(p+1)+n) \int_M f^{p+1}d\mu(t).$$
Since scalar curvature has a lower bound $-C_0$, then
\begin{equation}\label{eqn2}
\begin{split}\partial_t\int_Mf^{p+1}d\mu(t)&+\frac{1}{2}\left(\int_M(R+C_0)f^{p+1}d\mu(t)+4|\nabla f^{\frac{p+1}{2}}|^2d\mu(t)\right)\\
&\le [a(p+1)+n+C_0]\int_M f^{p+1}d\mu(t).
\end{split}
\end{equation}
For any $1>\sigma>\tau>0$, let
 \[\psi(t) = \left\{\begin{array}{l}
 0, ~~0\le t\le \tau \\
 \frac{t-\tau}{\sigma-\tau},~~\tau\le t \le \sigma \\
 1, ~~\sigma\le t\le 1\\
 \end{array} \right.\]
Multiplying (\ref{eqn2}) by $\psi$, we obtain
\begin{equation*}
\begin{split}\partial_t\left(\psi\int_Mf^{p+1}d\mu(t)\right)+\frac{1}{2}\psi\left(\int_M(R+C_0)f^{p+1}d\mu(t)+4|\nabla f^{\frac{p+1}{2}}|^2d\mu(t)\right)\\
\le [(a(p+1)+n+C_0)\psi+\psi '] \int_M f^{p+1}d\mu(t).
\end{split}
\end{equation*}
Integrating this with respect to $t$ we get
\begin{equation*}
\begin{split}\sup_{\sigma \le t\le 1}\int_{M}f^{p+1}d\mu(t)+\frac{1}{2}\left(\int_\sigma^1\int_M[(R+C_0)f^{p+1}d\mu(t)+4|\nabla f^{\frac{p+1}{2}}|^2]d\mu(t)dt\right)\\
\le [(a(p+1)+n+C_0)+\frac{1}{\sigma-\tau}] \int_\tau^1\int_M f^{p+1}d\mu(t)dt.
\end{split}
\end{equation*}
Applying the $L^2$ Sobolev inequality along the K\"{a}hler Ricci flow (Lemma \ref{Sobolev1}), we deduce
\begin{equation*}
\begin{split}&\int_\sigma^1\int_M f^{(p+1)(1+\frac{2}{n})}d\mu(t)dt\\
&\le\int_\sigma^1\left(\int_M f^{p+1}d\mu(t)\right)^{\frac{2}{n}}\left(\int_M f^{(p+1)\frac{n}{n-2}}d\mu(t)\right)^{\frac{n-2}{n}}dt\\
&\le \sup_{\sigma \le t\le 1}\left(\int_M f^{p+1}d\mu(t)\right)^{\frac{2}{n}}\int_\sigma^1A\left(\int_M[(R+C_0)f^{p+1}d\mu(t)+4|\nabla f^{\frac{p+1}{2}}|^2]d\mu(t)\right)dt\\
&\le 2A[(p+1)a+n+C_0+\frac{1}{\sigma-\tau}]^{1+\frac{2}{n}}\left(\int_\tau^1\int_M f^{p+1}d\mu(t)dt\right)^{1+\frac{2}{n}}.
\end{split}
\end{equation*}
We put
$$H(p,\tau)=\left(\int_\tau^1\int_Mf^pd\mu(t)dt\right)^\frac{1}{p},~~for~ any ~p\ge 2,0<\tau<1.$$
Thus
\begin{equation*}
H(p(1+\frac{2}{n}),\sigma)\le (2A)^{\frac{1}{p(1+\frac{2}{n})}}[pa+n+C_0+\frac{1}{\sigma-\tau}]^{\frac{1}{p}}H(p,\tau).
\end{equation*}
Fix $0<t_0<t_1<1,p_0\ge 2$ and set $\chi=1+\frac{2}{n},~p_k=p_0\chi^k,~\tau_k=t_0+(1-\frac{1}{\chi^k})(t_1-t_0)$. Then we have
\begin{equation*}
H(p_{k+1},\tau_{k+1})\le (2A)^{\frac{1}{p_{k+1}}}[p_ka+n+C_0+\frac{1}{t_1-t_0}\frac{\chi}{\chi-1}\chi^k]^{\frac{1}{p_k}}H(p_k,\tau_k).
\end{equation*}
By iteration, we have
\begin{equation*}
\begin{split}
&H(p_{m+1},\tau_{m+1})\\
&\le(2A)^{\sum_{k=0}^m\frac{1}{p_{k+1}}}[p_0a+n+C_0+\frac{1}{t_1-t_0}\frac{\chi}{\chi-1}]^{\sum_{k=0}^m\frac{1}{p_k}}
\chi^{\sum_{k=0}^m\frac{k}{p_k}}H(p_0,\tau_0).
\end{split}
\end{equation*}
Letting $m\to \infty$, we obtain
\begin{equation*}
H(p_{\infty},\tau_{\infty})\le C_1[p_0a+n+C_0+\frac{n+2}{2(t_1-t_0)}]^{\frac{n+2}{2p_0}}H(p_0,\tau_0)
\end{equation*}
for all $p_0\ge 2$. That is
\begin{equation*}
\sup_{(x,t)\in M\times[t_1,1]}|f(x,t)|\le C_1[p_0a+n+C_0+\frac{n+2}{2(t_1-t_0)}]^{\frac{n+2}{2p_0}}\left(\int_{t_0}^1\int_Mf^{p_0}d\mu(t)dt\right)^\frac{1}{p_0}.
\end{equation*}
Since $0<t_0<t_1<1$, we have
\begin{equation*}
\sup_{(x,t)\in M\times[t_1,1]}|f(x,t)|\le \frac{C_2}{(t_1-t_0)^{\frac{n+2}{2p_0}}}\left(\int_{t_0}^1\int_Mf^{p_0}d\mu(t)dt\right)^\frac{1}{p_0}, for~all~p_0\ge 2,
\end{equation*}
where $C_2$ depending only on $p_0,a$, dimension $n$, Sobolev constant $C_s$, volume $V$ and a lower bound of $R(g(0))$.
For $0<p<2$, we set
$$h(s)=\sup_{(x,t)\in M\times[s,1]}|f(x,t)|.$$
So
\begin{equation*}
\begin{split}
h(t_1)&\le \frac{C_2}{(t_1-t_0)^{\frac{n+2}{4}}}\left(\int_{t_0}^1\int_Mf^{2}d\mu(t)dt\right)^\frac{1}{2}\\
&\le h(t_0)^{\frac{1}{2}(2-p)}\frac{C_2}{(t_1-t_0)^{\frac{n+2}{4}}}\left(\int_{t_0}^1\int_Mf^{p}d\mu(t)dt\right)^\frac{1}{2}\\
&\le\frac{1}{2}h(t_0)+\frac{C_3}{(t_1-t_0)^{\frac{n+2}{2p}}}\left(\int_{t_0}^1\int_Mf^{p}d\mu(t)dt\right)^\frac{1}{p}.
\end{split}
\end{equation*}
By a iteration lemma(Lemma 4.3 of \cite{HL}), we get
\begin{equation*}
h(t_1)\le\frac{C_4}{(t_1-t_0)^{\frac{n+2}{2p}}}\left(\int_{0}^1\int_Mf^{p}d\mu(t)dt\right)^\frac{1}{p}, for~all~0<t_0<t_1<1,p>0.
\end{equation*}
Letting $t_0\to 0$ we have
\begin{equation*}
h(t_1)\le\frac{C_4}{t_1^{\frac{n+2}{2p}}}\left(\int_{0}^1\int_Mf^{p}d\mu(t)dt\right)^\frac{1}{p},
\end{equation*}
where constant $C_4$ depends only on $p,~a$, dimension $n$, Sobolev constant $C_s$, volume $V$ and a lower bound of $R(0)$.
\\
\noindent Case $n=2$: we only need to replace Lemma \ref{Sobolev1} with Lemma \ref{Sobolev2}.\qed
\begin{remark}
From the above proof, one can find that it also holds for Ricci flow on real Riemannian manifold.
\end{remark}

\section{Proof of Theorem \ref{main} }\label{secThm2}
The purpose of this section is to prove Theorem \ref{main}. This mainly bases on the linear parabolic estimate in Theorem \ref{thm1}. Since along the K\"{a}hler Ricci flow (\ref{Kahler}) we have
the evolution equations, see\cite{ST}
$$(\partial_t-\Delta)R=|\nabla\bar{\nabla}u|^2+R-\frac{n}{2},$$
$$(\partial_t-\Delta)|\nabla u|^2=-|\nabla\nabla u|^2-|\nabla\bar{\nabla}u|^2+|\nabla u|^2.$$
Hence
\begin{equation*}
(\partial_t-\Delta)(R+|\nabla u|^2)=-|\nabla\nabla u|^2+R+|\nabla u|^2-\frac{n}{2}\le R+|\nabla u|^2.
\end{equation*}
Applying Theorem \ref{thm1}, we have
\begin{equation}\label{eqn4}
\sup_{x\in M}|R+|\nabla u|^2|(x,t)\le \frac{C}{t^{\frac{n_0+2}{2}}}\int_0^1\int_M|R+|\nabla u|^2|d\mu(t)dt.
\end{equation}
for all $0<t<1$, where constant $C$ depends only on dimension $n$, Sobolev constant $C_s$ of $g(0)$ and a lower bound of $R(g(0))$. Since Ricci curvature lower bound $-K$, volume lower bound $V_0$ and diameter upper bound $d$ can deduce the $L^2$ Sobolev inequality(see Theorem {{3.2}} of \cite{H} or \cite{L}). So $C$ depends only on $n,~K,~V_0$ and $d$.

In order to get an upper bound of $R$, it suffices to estimate $\int_0^1\int_M|R+|\nabla u|^2|d\mu(t)dt$.\\
Since $\Delta u=\frac{n}{2}-R\le \frac{n}{2}+C_0$, then
\begin{equation*}
\begin{split}
\int_M|\Delta u|d\mu(t)&=\int_M|\Delta u-(\frac{n}{2}+C_0)+(\frac{n}{2}+C_0)|d\mu(t)\\
&\le\int_M(\frac{n}{2}+C_0)-\Delta u+(\frac{n}{2}+C_0)d\mu(t)\\
&=(n+2C_0)V.
\end{split}
\end{equation*}
Hence $\int_M |R|d\mu(t)\le C_1V$, and then
\begin{equation}\label{scalar}
\int_0^1\int_M|R|d\mu(t)dt\le C_1V.
\end{equation}
On the other hand, normalize $u$ by $\int_Me^{-u}d\mu(t)=(2\pi)^\frac{n}{2}.$
The evolution equation of $u$ is given by (\ref{Peqa}),
$$\partial_t u=\Delta u+u+a, ~~~a=-\int_Mue^{-u}(2\pi)^{\frac{n}{2}}\le C'.$$
where $C'$ depends only on dimension $n$, volume of $g(0)$, Sobolev constant $C_s$ of $g(0)$ and a lower bound of $R(g(0))$.
Thus
\begin{equation}\label{eqn8}
\begin{split}
\partial_t\int_Mud\mu(t)&=\int_M\partial_tud\mu(t)+\int_Mu\partial_td\mu(t)\\
&=\int_Mud\mu(t)+a V+\int_Mu\Delta ud\mu(t)\\
&=\int_Mud\mu(t)+a V-\int_M|\nabla u|^2d\mu(t)\\
&\le\int_Mud\mu(t)+C'V
\end{split}
\end{equation}
and
\begin{equation}\label{eqn5}
\int_Mud\mu(t)\le \left(\int_Mud\mu(0)+C'V\right)e^t-CV\le\left(\int_Mud\mu(0)+C'V\right)e^t.
\end{equation}
Integrating (\ref{eqn8}) over $[0,1]$, we have
\begin{equation*}
\int_0^1\partial_t\int_Mud\mu(t)dt=\int_0^1\int_Mud\mu(t)dt+a V-\int_0^1\int_M|\nabla u|^2d\mu(t)dt.
\end{equation*}
By (\ref{eqn5}) we have
\begin{equation*}
\begin{split}
&\int_0^1\int_M|\nabla u|^2d\mu(t)dt\\
&\le \int_0^1\int_Mud\mu(t)dt+aV +\int_Mud\mu(0)-\int_Mud\mu(1)\\
&\le C'V+(e-1)\left(\int_Mud\mu(0)+C'V\right)+\int_Mud\mu(0)-\int_Mud\mu(1).
\end{split}
\end{equation*}
Since $\int_Me^{-u}d\mu(1)=(2\pi)^{\frac{n}{2}}$, by Jensen's inequality
$$\ln\frac{(2\pi)^{\frac{n}{2}}}{V}= \ln\left(\int_Me^{-u}\frac{d\mu(1)}{V}\right)\ge \frac{1}{V}\int_M-ud\mu(1).$$
Hence we get
\begin{equation}\label{eqn3}
\begin{split}
\int_0^1\int_M|\nabla u|^2d\mu(t)dt&\le C'V+(e-1)\left(\int_Mud\mu(0)+C'V\right)+\int_Mud\mu(0)+V\ln\frac{(2\pi)^{\frac{n}{2}}}{V}\\
&\le e\left(\int_Mud\mu(0)+C'V\right)+V\ln\frac{(2\pi)^{\frac{n}{2}}}{V}.
\end{split}
\end{equation}
Combining (\ref{scalar}) with (\ref{eqn3}), we arrive at
\begin{equation}\label{eqn6}
\int_0^1\int_M|R+|\nabla u|^2|d\mu(t)dt\le C_1V\,+\,e\left(\int_Mud\mu(0)+C'V\right)\,+\,V\ln\frac{(2\pi)^{\frac{n}{2}}}{V}.
\end{equation}
In order to  deduce an upper bound of scalar curvature $R$ and gradient bound of Ricci potential $u$, it suffices to estimate the upper bound of $\int_Mud\mu(0)$.

Actually, since we have a lower bound of Ricci curvature $\Ric(g(0))\ge-Kg(0)$, volume lower bound $V_0>0$ and diameter upper bound $d$ at time $t=0$, we can get a lower bound of the Green function $\Gamma (x,y)\ge -B$ at time $t=0$ (see \cite{BM} and \cite{CL}), where $B$ depends only on $K,~V_0,~d$ and $n$. At time $t=0$,
 since $\int_Me^{-u}d\mu(0)=(2\pi)^{\frac{n}{2}}$, there must be a point $x_0$ such that $u(x_0)\le -\ln\frac{(2\pi)^{\frac{n}{2}}}{V}$. Then by Green's formula
$$u(x_0)=\frac{1}{V}\int_Mud\mu(0)+\int_M\Gamma(x_0,y)(-\Delta u(y))d\mu(0),$$
that is
\begin{equation}\label{eqn10}
\begin{split}
\frac{1}{V}\int_Mud\mu(0)&=u(x_0)+\int_M\Gamma(x_0,y)\Delta u(y)d\mu(0)\\
&=u(x_0)+\int_M(\Gamma(x_0,y)+B)\Delta u(y)d\mu(0)\\
&=u(x_0)+\int_M(\Gamma(x_0,y)+B)(\frac{n}{2}-R)d\mu(0)\\
&\le u(x_0)+(\frac{n}{2}+C_0)\int_M\Gamma(x_0,y)+Bd\mu(0)\\
&\le-\ln\frac{(2\pi)^{\frac{n}{2}}}{V}+(\frac{n}{2}+C_0)BV.
\end{split}
\end{equation}
Substituting (\ref{eqn10}) into (\ref{eqn6}), we have
\begin{equation*}
\int_0^1\int_M|R+|\nabla u|^2|d\mu(t)dt\le C_2
\end{equation*}
where constant $C_2$ depends only on $K,~n,~V_0$ and $d$. Hence we finish the proof.\qed

\section{Partial $C^0$ estimate on $S^2$}\label{secPCE}
In this section, suppose $(M,\omega)$ is a complex dimension $1$ Fano manifold with $R(\omega)\ge R_0>0,~ \mbox{\Vol}(M,\omega)\ge {V_0}>0$. Consider the heat flow(see \cite{T2})
\begin{equation}\label{eqn9}
\frac{\partial f}{\partial s}=\log \left(\frac{(\omega+\partial \overline{{\partial}}f)}{\omega}\right)+f-h_{\omega},~f|_{s=0}=0.
\end{equation}
This is in fact the K\"{a}hler Ricci flow. Here $h_\omega$ is Ricci potential of $\omega$. We will
denote by $f_s$ , $\omega_s$ and $R_s$, the function $f(s,\cdot)$, the K\"{a}hler form $\omega+\partial \overline{{\partial}}f_s$ and scalar curvature $R(\omega_s)$.
\begin{remark}
In this section, the following constants $C$ depend only on $R_0$ and $V_0$. It maybe change line by line. Moreover, constants $C_p, c_s$ will depend on $p$ or $s$.
\end{remark}

Differentiating (\ref{eqn9}), we obtain
\begin{equation*}
\partial \overline{{\partial}}\left(\frac{\partial f_s}{\partial s}\right)=-\Ric(\omega_s)+\omega_s
\end{equation*}
This implies that the Ricci potential $h_{\omega_s}=-\frac{\partial f_s}{\partial s}+c_s$ where
$c_s$ is constant. Since $f_0=0$, we have $c_0=0$.
\begin{lemma}\label{lbe}
There exist constants $V_0^{'},~d$ depending only on $R_0,~V_0$ such that, for all $s>0$, holds
\begin{equation*}
R(\omega_s)\ge0,~V_0\le \mbox{\Vol}(M,\omega_s)\le V_0^{'},~\diam(M,\omega_s)\le d.
\end{equation*}
What's more, for all $s\ge R_0>0$, $R(\omega_s)$ has a uniform upper bound.
\end{lemma}
\begin{proof}
We have the evolution equation along the K\"{a}hler Ricci flow,
\begin{equation*}
\frac{\partial R_s}{\partial s}=\Delta_s R_s+R_s^2-R_s,
\end{equation*}
Since $R(\omega)\ge R_0>0$, by maximum principle, $R_s\ge 0$. On the other hand, $R^2_s-R_s\ge -\frac{1}{4}$, then
\begin{equation*}
\frac{\partial R_s}{\partial s}\ge\Delta_s R_s-\frac{1}{4}.
\end{equation*}
Applying maximum principle again, we obtain
\begin{equation*}
R(\omega_s)_{min}\ge R_0-\frac{1}{4}s
\end{equation*}
for time $0\le s\le R_0$. Then we have $R(\omega_s)\ge \frac{3}{4}R_0>0$ for time $0\le s\le R_0$. By Meyer's diameter theorem, we obtain the upper bound of diameter. By volume comparison theorem and the upper bound of diameter, we deduce the upper bound of volume.

For $s\ge R_0$, applying Theorem \ref{main}, we have an upper bound of $R(\omega_{R_0})$ and an upper bound of the $C^1$-norm for $h_{\omega_{R_0}}$. Using Perelman's theorem, i.e., Theorem \ref{Perelman}, we get the upper bound of $R(\omega_{s})$ and \diam($M,\omega_s$) for all $s\ge R_0$.
\end{proof}
Normalize $h_{\omega_s}$ by $\int_M e^{h_{\omega_s}} \omega_s=2\pi$. Then along the K\"{a}hler Ricci flow, we have the evolution equation of $h_{\omega_s}$, see \cite{ST}
\begin{equation}\label{eqn11}
\frac{\partial h_{\omega_s}}{\partial s}=\Delta_s h_{\omega_s}+h_{\omega_s}-a,
\end{equation}
where $a\,=\,\frac{1}{2\pi}\int_M h_{\omega_s}e^{h_{\omega_s}}\omega_s$ is uniformly bounded.
\begin{lemma}\label{lpe}
$\forall p\ge 1$, $\forall s\ge 0$, there exists constant $C_p$ depending only on $p,~R_0$ and $V_0$ such that
\begin{equation*}
||h_{\omega_s}||_{L^p(M,\omega_s)}\le C_p.
\end{equation*}
\end{lemma}
\begin{proof}
Since $-\partial \overline{{\partial}}h_{\omega_s}=-\Ric(\omega_s)+\omega_s$, by taking the trace, we get
\begin{equation}\label{eqn14}
-\Delta_s h_{\omega_s}=1-R_s.
\end{equation}
By Lemma \ref{lbe}, we can estimate the Green function $\Gamma_s(x,y)$ of $\Delta_s$ at all time $s$(see also \cite{BM},\cite{CL},\cite{Gr}),
\begin{equation}\label{greenb}
-C<\Gamma_s(x,y)\le C |\log(d(x,y))|+C
\end{equation}
where constant $C$ depends only on $R_0,~V_0$.

Since $\int_M e^{h_{\omega_s}}\omega_s=2\pi$, by Jensen's inequality, we have $\int_M h_{\omega_s}\omega_s\le C$. On the
other hand, by the lower bound estimate of $\Gamma_s(x,y)$, we can estimate the lower bound of $\int_M h_{\omega_s}\omega_s$ . Actually, since
$\int_M e^{h_{\omega_s}}\omega_s=2\pi$, there must be a point $x_s$ such that $h_{\omega_s}(x_s)=\log(\frac{2\pi}{V})$. Then by Green's formula
\begin{equation*}
\begin{split}
\frac{1}{V}\int_M h_{\omega_s}\omega &=h_{\omega_s}(x_s) +\int_M \Gamma_s(x,y)\Delta_s h_{\omega_s}\omega_s\\
&=h_{\omega_s}(x_s) +\int_M (\Gamma_s(x,y)+C)(R_s-1)\omega_s\\
&\ge h_{\omega_s}(x_s) -\int_M (\Gamma_s(x,y)+C)\omega_s\\
&= \log(\frac{2\pi}{V})-CV.
\end{split}
\end{equation*}
Hence, there exists a uniform constant $C$ such that the following holds
\begin{equation*}
|\frac{1}{V}\int_Mh_{\omega_s}\omega_s|\le C.
\end{equation*}
Using Green's formula again, we have
\begin{equation}\label{eqn15}
h_{\omega_s}(x)=\frac{1}{V}\int_Mh_{\omega_s}\omega_s+\int_M\Gamma_s(x,y)(-\Delta_s h_{\omega_s}(y))\omega_s(y).
\end{equation}
Moreover, $||\Delta_s h_{\omega_s}||_{L^1(M,\omega_s)}$ is uniformly bounded. By (\ref{eqn14}) and Gauss-Bonnet theorem,
\begin{equation*}
||\Delta_s h_{\omega_s}||_{L^1(M,\omega_s)}\le V_0^{'}+\int_M R_s\omega_s\le C.
\end{equation*}
Applying Young's inequality to (\ref{eqn15}), we arrive at
\begin{equation*}
||h_{\omega_s}||_{L^p(M,\omega_s)}\le C+\sup_{x\in M}\left(\int_M |\Gamma_s(x,y)|^p\omega_s(y)\right)^{\frac{1}{p}}||\Delta_s h_{\omega_s}||_{L^1(M,\omega_s)}.
\end{equation*}
Claim: $\int_M |\Gamma_s(x,y)|^p\omega_s(y)\le C_p$, for some constant $C_p$ depending only on $p,~R_0,~V_0$.\\

Actually, by inequality (\ref{greenb}), we only need to show $\int_M|\log(d(x,y))|^p\omega_s(y)$ is uniformly bounded.
\begin{equation*}
\begin{split}
\int_M|\log(d(x,y))|^p\omega_s(y)&\le\int_{B(x,1)}|\log(d(x,y))|^p\omega_s(y)+\int_{M\backslash B(x,1)}|\log(d(x,y))|^p\omega_s(y)\\
&\le \sum_{k=0}^{\infty}\int_{B(x,\frac{1}{2^k})\backslash B(x,\frac{1}{2^{k+1}})}|\log(d(x,y))|^p\omega_s(y)+V_0^{'}\log(d)^p\\
&\le \sum_{k=0}^{\infty}|\log(\frac{1}{2^{k+1}})|^p\int_{B(x,\frac{1}{2^k})}\omega_s(y)+V_0^{'}\log(d)^p\\
&\le \sum_{k=0}^{\infty}(\log({2})(k+1))^p \mbox{vol}(B(x,\frac{1}{2^k}))+V_0^{'}\log(d)^p\\
&\le \sum_{k=0}^{\infty}(\log({2})(k+1))^p \frac{c}{4^k}+V_0^{'}\log(d)^p\\
&\le C_p
\end{split}
\end{equation*}
where we have used the volume comparison, so the Claim holds. Hence the $L^p$-norm of $h_{\omega_s}$ is uniformly bounded.
\end{proof}
\begin{lemma}\label{csb}
There exists a constant $C$ depending only on $R_0,V_0$ such that
\begin{equation}
|c_s|\le C(e^s-1),~for~all~s>0
\end{equation}
where $c_s=\frac{\partial f_s}{\partial s}+h_{\omega_s}$.
\end{lemma}
\begin{proof}
Differentiating (\ref{eqn9}) we have
\begin{equation}\label{eqn12}
\frac{\partial}{\partial s}\left(\frac{\partial f_s}{\partial s}\right)=\Delta_s \left(\frac{\partial f_s}{\partial s}\right)+\frac{\partial f_s}{\partial s}
\end{equation}
 and combine this with (\ref{eqn11}). We have
\begin{equation*}
\frac{\partial}{\partial s}\left(\frac{\partial f_s}{\partial s}+h_{\omega_s}\right)=\Delta_s\left(\frac{\partial f_s}{\partial s}+h_{\omega_s}\right)+\left(\frac{\partial f_s}{\partial s}+h_{\omega_s}\right)-a.
\end{equation*}
That is
\begin{equation*}
\frac{\partial}{\partial s}c_s=c_s-a.
\end{equation*}
Since $c_0=0$, $a$ is bounded, we can deduce the bound for $c_s$ easily,
\begin{equation*}
|c_s|\le C(e^s-1).
\end{equation*}
\end{proof}
Now we can estimate the K\"{a}hler potential,
\begin{lemma}\label{fsb}
There exists a constant $C$ depending only on $R_0,~V_0$ such that
\begin{equation*}
|f_s(x)|\le C\sqrt{s}
\end{equation*}
for all $s\le 1.$
\end{lemma}
\begin{proof}
By Lemma \ref{lpe} and Lemma \ref{csb}, we have
\begin{equation*}
||\frac{\partial f_s}{\partial s}||_{L^p}\le C_p, ~for~all~s\le 1.
\end{equation*}
Combining Theorem \ref{thm1} with equation (\ref{eqn12}), we obtain
\begin{equation*}
|\frac{\partial f_s}{\partial s}|\le \frac{C_{p,n_0}}{s^{\frac{n_0+2}{2p}}},~for~all~s\le 1.
\end{equation*}
Choose $n_0=3,~p=5$. We have $|\frac{\partial f_s}{\partial s}|\le \frac{C}{\sqrt{s}}$. Then
\begin{equation*}
|f_t(x)-f_0(x)|\le \int_0^t|\frac{\partial f_s}{\partial s}|ds\le C \sqrt{t}
\end{equation*}
for all $t\le 1$. Noting that $f_0=0$, then we have  $|f_s(x)|\le C\sqrt{s}, $~for~all~$s\le 1$.
\end{proof}
Since
\begin{equation*}
\Ric(\omega_s)-\partial \overline{{\partial}}h_{\omega_s}=\omega_s,
\end{equation*}
we can choose $\omega_se^{h_{\omega_s}}$ as a Hermitian metric of anti-canonical line bundle $K_M^{-1}$ with curvature form $\omega_s$, denoting by $H_{\omega_s}$.
\begin{lemma}\label{hequ}
For all $s\le 1$,$H_{\omega}$ and $H_{\omega_s}$ are equivalent. i.e., there exists a constant $C$ depending only on $R_0,V_0$ such that
\begin{equation*}
\frac{1}{C}H_{\omega}\le H_{\omega_s}\le C H_{\omega}.
\end{equation*}
\end{lemma}
\begin{proof}
By equation (\ref{eqn9}), we have
\begin{equation*}
e^{\frac{\partial f_s}{\partial s}-f_s+h_\omega}=\frac{\omega_s}{\omega}.
\end{equation*}
Thus
\begin{equation*}
e^{c_s-f_s}=\frac{\omega_s e^{-\frac{\partial f_s}{\partial s}+c_s}}{\omega e^{h_\omega}}=\frac{H_{\omega_s}}{H_\omega}.
\end{equation*}
By Lemma \ref{csb} and Lemma \ref{fsb}, we conclude
\begin{equation*}
\frac{1}{C}\le\frac{H_{\omega_s}}{H_\omega}\le C
\end{equation*}
for all $s\le 1.$
\end{proof}
Now we turn to prove Theorem \ref{thm2}.\\
\noindent \textit{Proof of Theorem \ref{thm2}.} We will argue by
contradiction. Suppose there exists a family Fano manifolds $(M^i,\omega^i)$ satisfying $R(\omega^i)\ge R_0,~\mbox{\Vol}(M^i,\omega^i)\ge V_0$, but there exist $\{x_i\in M^i\}$ such that
\begin{equation*}
\rho_{\omega^i,l}(x_i)\to 0 ,~\mbox{when}~i\to \infty.
\end{equation*}
Assume $S^i\in H^0(M^i,K^{-l}_{M^i})$ satisfying  $||S^i||^2_{H^{\otimes l}_{\omega_1^i},\omega_1^i}=1$,
\begin{equation*}
||S^i||_{H_{\omega_1^i}}^2(x_i)=
\sup_{||S||^2_{H^{\otimes l}_{\omega_1^i},\omega_1^i}=1}||S||^2_{H^{\otimes l}_{\omega_1^i}}(x_i)=\eta_{\omega_1^i}(x_i).
\end{equation*}
Here $\omega_1^i$ is the K\"{a}hler form at time $t=1$ along the K\"{a}hler Ricci flow with initial metric $\omega^i$ on Fano manifold $M^i$. Then by Lemma \ref{hequ} and Remark \ref{rhoremark},
\begin{equation*}
\begin{split}
\eta_{\omega_1^i}(x_i)&=\frac{||S^i||_{H^{\otimes l}_{\omega_1^i}}^2(x_i)}{||S^i||_{H^{\otimes l}_{\omega^i}}^2(x_i)}\cdot||S||^2_{H^{\otimes l}_{\omega^i},\omega^i}
\cdot\frac{||S^i||_{H^{\otimes l}_{\omega^i}}^2(x_i)}{||S||^2_{H^{\otimes l}_{\omega^i},\omega^i}}\\
&\le \frac{||S^i||_{H^{\otimes l}_{\omega_1^i}}^2(x_i)}{||S^i||_{H^{\otimes l}_{\omega^i}}^2(x_i)}\cdot||S||^2_{H^{\otimes l}_{\omega^i},\omega^i}\cdot \eta_{\omega^i}(x_i)\\
&\le C||S||^2_{H^{\otimes l}_{\omega^i},\omega^i}\cdot \rho_{\omega^i,l}(x_i).
\end{split}
\end{equation*}
Using Remark \ref{rhoremark} again, we obtain
\begin{equation*}\label{eqn13}
\rho_{\omega^i_1,l}(x_i)\le NC||S||^2_{H^{\otimes l}_{\omega^i},\omega^i}\cdot\rho_{\omega^i,l}(x_i).
\end{equation*}
By Lemma \ref{lbe}, we have $0\le R(\omega^i_1)\le C,~{V_0}\le\mbox{\Vol}(M^i,\omega^i_1)\le V_0^{'}$. Then by Moser's iteration we have $||S^i||^2_{H^{\otimes l}_{\omega^i_1}}(x)\le C,$ for all $x\in M^i$(see\cite{T1}). Since $H^{\otimes l}_{\omega_1^i}$ and $H^{\otimes l}_{\omega^i}$ are equivalent, so
\begin{equation*}
||S^i||^2_{H^{\otimes l}_{\omega^i}}(x)\le C
\end{equation*}
Thus
\begin{equation*}
||S||^2_{H^{\otimes l}_{\omega^i},\omega^i}\le CV_0^{'}
\end{equation*}
Hence
\begin{equation*}
\rho_{\omega^i_1,l}(x_i)\le NC\rho_{\omega^i,l}(x_i)\to 0,~\mbox{when} ~i\to \infty.
\end{equation*}
But the family $(M^i,\omega^i_1)$ have bounded curvature, bounded volume, bounded diameter. By Hamilton's compactness theorem for Ricci flow(\cite{Ha}), $\rho_{\omega^i_1,l}(x)$ have a uniform lower bound. This is a contradiction. Hence we complete the proof.\qed
\begin{remark}
Under the above estimates, we can also prove that the Bergman kernel is uniformly continuous along the K\"{a}hler Ricci flow
by showing the measure is uniformly continuous.
\end{remark}

\section{Partial $C^0$ estimate for complex dimension $\ge 2$}\label{secPCEH}
Suppose $(M,\omega)$ is a complex dimension $m=\frac{n}{2}\ge 2$ Fano manifold, with $\Ric(\omega)\ge R_0, ~\mbox{\Vol}(M,\omega)\ge {V_0}$, $\mbox{\diam}(M,\omega)\le d$. Consider the heat flow(see \cite{T2})
\begin{equation}\label{mfs}
\frac{\partial f}{\partial s}=\log \left(\frac{(\omega+\partial \overline{{\partial}}f)^m}{\omega^m}\right)+f-h_{\omega},~f|_{s=0}=0.
\end{equation}
This is in fact the K\"{a}hler Ricci flow. Here $h_\omega$ is Ricci potential of $\omega$. We will
denote by $f_s$ and $\omega_s$ the function $f(s,\cdot)$ and the K\"{a}hler form $\omega+\partial \overline{{\partial}}f_s$.
\begin{remark}
In this section, the following constants $C$ depend only on $R_0,~V_0,~d,~n$. It maybe change line by line.
\end{remark}
\begin{remark}
In order to estimate the Bergman kernel for high dimension, we only need to estimate the K\"{a}hler potential like the case of complex dimension $1$.
However, the estimates of dimension $1$ mostly depend on the property of dimension $1$ which we can't extend to high dimension. Hence we must need another approach.
\end{remark}
First, we can estimate the upper bound of $h_{\omega}$. The proof is similar to Lemma \ref{Plem1}.
\begin{lemma}\label{mhwu}
Suppose $h_{\omega}$ is the Ricci potential of $\omega$. Normalize it by $\int_M{e^{h_{\omega}}}\omega^m=(2\pi)^{{m}}$. Then there is a uniform constant $C$ depending only on $R_0,~V_0,~d$ and $m$ such that
\begin{equation}
h_\omega\le C,
\end{equation}
\end{lemma}
\begin{proof}
Since $\Ric(\omega)-\partial \overline{{\partial}}h_\omega=\omega$, taking the trace respect to $\omega$, we obtain
\begin{equation*}
\Delta h_\omega=R-m.
\end{equation*}
Thus
\begin{equation}\label{mehw}
\begin{split}
\Delta e^{h_\omega}&=\Delta h_\omega e^{h_\omega}+|\nabla h_\omega|^2e^{h_\omega}\\
&\ge (R-m)e^{h_\omega}\\
&\ge -C e^{h_\omega}.
\end{split}
\end{equation}
Moreover, since $\Ric(\omega)\ge R_0,~ \mbox{\Vol}(M,\omega)\ge {V_0}$, $\mbox{\diam}(M,\omega)\le d$, then the $L^2$-Sobolev inequality holds,
\begin{equation*}
\left(\int_M\phi^{\frac{2m}{m-1}}\omega^m\right)^{\frac{m-1}{m}}\le C\left(\int_M|\nabla \phi|^2\omega^m+\int_M\phi^2\omega^m\right).
\end{equation*}
The Moser's iteration and inequality (\ref{mehw}) imply
\begin{equation*}
e^{h_\omega}\le C\int_M e^{h_\omega}\omega^m=C (2\pi)^m.
\end{equation*}
Hence the upper bound of $h_\omega$ follows.
\end{proof}
Now we can estimate the K\"{a}hler potential $f_s$,
\begin{lemma}\label{mfsl}
Let $f_s$ be the K\"{a}hler potential of $\omega_s$, satisfying equation (\ref{mfs}). Then for all $s\ge 0$, we have a lower bound estimate for $f_s$,
\begin{equation}
f_s\ge C(1-e^s).
\end{equation}
\end{lemma}
\begin{proof}
By Lemma \ref{mhwu},
\begin{equation*}
\begin{split}
\frac{\partial f_s}{\partial s}&=\log \left(\frac{\omega_s^m}{\omega^m}\right)+f_s-h_\omega\\
&\ge \log \left(\frac{\omega_s^m}{\omega^m}\right)+f_s-C.
\end{split}
\end{equation*}
Using maximum principle, we have
\begin{equation*}
\frac{\partial f_s}{\partial s}\ge f_s-C.
\end{equation*}
Then
\begin{equation*}
\partial_s \left(f_s e^{-s}\right)\ge -C e^{-s}.
\end{equation*}
Since $f_0=0$, we deduce
\begin{equation*}
f_s\ge C(1-e^s).
\end{equation*}
\end{proof}
\begin{lemma}\label{mfsu}
Let $f_s$ be the K\"{a}hler potential of $\omega_s$, satisfying equation (\ref{mfs}). Then for all $s\ge 0$, we have an upper bound estimate for $f_s$,
\begin{equation}
f_s\le Ce^s.
\end{equation}
\end{lemma}
\begin{proof}
Since $\omega_s =\omega+\partial \overline{\partial}f_s>0$, taking trace respect to $\omega$, we have
\begin{equation*}
m+\Delta f_s>0
\end{equation*}
By the assumption of initial metric $\omega$, we can control the Green function lower bound $\Gamma (x,y)\ge -C$. Applying Green's formula, we have
\begin{equation}\label{efsu}
\begin{split}
f_s(x)&=\frac{1}{V}\int_Mf_s\omega^m+\int_M\Gamma (x,y)(-\Delta f_s)\omega^m\\
&= \frac{1}{V}\int_Mf_s\omega^m+\int_M(\Gamma (x,y)+C)(-\Delta f_s)\omega^m\\
&\le \frac{1}{V}\int_Mf_s\omega^m+m\int_M(\Gamma (x,y)+C)\omega^m\\
&\le mCV+\frac{1}{V}\int_Mf_s\omega^m.
\end{split}
\end{equation}
In order to get an upper bound of $f_s$, it suffices to estimate $\frac{1}{V}\int_Mf_s\omega^m$. Actually, by (\ref{mfs}) and Jensen's inequality,
\begin{equation}
\begin{split}
\frac{\partial}{\partial s}\left(\frac{1}{V}\int_Mf_s\omega^m\right)&=\frac{1}{V}\int_M\frac{\partial f_s}{\partial s}\omega^m\\
&=\frac{1}{V}\int_M\log \left(\frac{\omega_s^m}{\omega^m}\right)\omega^m+\frac{1}{V}\int_M f_s\omega^m-\frac{1}{V}\int_M h_\omega \omega^m\\
&\le \log \left(\int_M\frac{\omega_s^m}{\omega^m} \frac{\omega^m}{V}\right)+\frac{1}{V}\int_M f_s\omega^m-\frac{1}{V}\int_M h_\omega \omega^m\\
&\le \frac{1}{V}\int_M f_s\omega^m-\frac{1}{V}\int_M h_\omega \omega^m.
\end{split}
\end{equation}
On the other hand, we have a lower bound estimate of $\frac{1}{V}\int_M h_\omega \omega^m\ge -C$, see (\ref{eqn10}), noting that $-u=h_\omega$. Hence
\begin{equation*}
\frac{\partial}{\partial s}\left(\frac{1}{V}\int_Mf_s\omega^m\right)\le C+\frac{1}{V}\int_Mf_s\omega^m.
\end{equation*}
Since $f_0=0$, then we get the upper bound of $\frac{1}{V}\int_Mf_s\omega^m$,
\begin{equation*}
\frac{1}{V}\int_Mf_s\omega^m\le C(e^s-1).
\end{equation*}
Substituting into (\ref{efsu}), the lemma follows.
\end{proof}
\begin{remark}
Since we can estimate the $L^1$-norm of K\"{a}hler potential $f_s$, one can use Moser's iteration to get the upper bound of $f_s$. Furthermore the upper bound will go to zero when time go to zero.
\end{remark}
\begin{remark}If we have a $C^0$ bound of the Ricci potential $h_{\omega}$, by using Maximum principal theorem, one can easily deduce all the above bound(see \cite{T2}). However, we do not have the $C^0$ bound for the Ricci potential $h_{\omega}$ here.
\end{remark}

We can choose $\omega_s^me^{h_{\omega_s}}$ as a Hermitian metric of anti-canonical line bundle $K_M^{-1}$ with curvature form $\omega_s$, since
\begin{equation*}
\Ric(\omega_s)-\partial \overline{{\partial}}h_{\omega_s}=\omega_s.
\end{equation*}
Denote $\omega_s^me^{h_{\omega_s}}$ by $H_{\omega_s}$.
Noticing that the constant $c_s=h_{\omega_s}+\frac{\partial f_s}{\partial s}$ have a uniform bound in Lemma \ref{csb} and combining Lemma \ref{mfsl} with Lemma \ref{mfsu}, we have
\begin{lemma}\label{mhe}
For all $s\le 1$, $H_{\omega}$ and $H_{\omega_s}$ are equivalent. i.e.,
\begin{equation*}
\frac{1}{C}H_{\omega}\le H_{\omega_s}\le C H_{\omega}.
\end{equation*}
 where the constant $C$ depends only on $R_0,~V_0,~m$ and $d$.
\end{lemma}
\begin{proof}
By equation (\ref{mfs}), we have
\begin{equation*}
e^{\frac{\partial f_s}{\partial s}-f_s+h_\omega}=\frac{\omega_s^m}{\omega^m}.
\end{equation*}
Thus
\begin{equation*}
e^{c_s-f_s}=\frac{\omega_s^m e^{-\frac{\partial f_s}{\partial s}+c_s}}{\omega^m e^{h_\omega}}=\frac{H_{\omega_s}}{H_\omega}.
\end{equation*}
By Lemma \ref{csb}, Lemma \ref{mfsl}, and Lemma \ref{mfsu}, we deduce
\begin{equation*}
\frac{1}{C}\le\frac{H_{\omega_s}}{H_\omega}\le C
\end{equation*}
for all $s\le 1.$
\end{proof}
The diameter upper bound is also under control.
\begin{lemma}\label{diame}
For $\frac{1}{2}\le s\le 1$, there exits a uniform constant $D>0$ depending only on $R_0,~V_0,~d$ and $m$ such that
\begin{equation}
\mbox{\diam}(M,\omega_s)\le D.
\end{equation}
\end{lemma}
\begin{proof}
Since along the K\"{a}hler Ricci flow, the following $L^2$-Sobolev inequality holds(see Ye\cite{Y2} or Zhang\cite{Z}),
\begin{equation*}
\left(\int_{M}\phi(x)^{\frac{2m}{m-1}}\omega^m_s\right)^{\frac{m-1}{m}}\le A\left(\int_M \left[4|\nabla \phi(x)|^2+(R_s+C_0)\phi^2(x)\right]\omega^m_s\right)
\end{equation*}
for all $0\le s\le 1$ and $\phi \in W^{1,2}(M,\omega_s)$.
For $\frac{1}{2}\le s\le 1$, by Theorem \ref{main}, we have $|R_s|\le C$. Hence
\begin{equation*}
\left(\int_{M}\phi(x)^{\frac{2m}{m-1}}\omega^m_s\right)^{\frac{m-1}{m}}\le C\left(\int_M \left[4|\nabla \phi(x)|^2+\phi^2(x)\right]\omega^m_s\right).
\end{equation*}
Since the $L^2$-Sobolev inequality implies the non-collapsing of volume (see Lemma 2.2 of \cite{H}) and
the volume is preserved along the K\"{a}hler Ricci flow, thus there must be a uniform upper bound for diameter.
\end{proof}
\begin{theorem}\label{Beq}
For all $\frac{1}{2}\le s\le 1$, all $l\ge 1$, the Bergman kernels $\rho_{\omega,l}$ and $\rho_{\omega_s,l}$ are equivalent. i.e., there exists a constant $C_l$ depending only on $R_0,~V_0,~d,~l$ and $m$, such that
\begin{equation}
\frac{1}{C_l}\rho_{\omega,l}\le \rho_{\omega_s,l}\le C_l\rho_{\omega,l}.
\end{equation}
\end{theorem}
\begin{proof}
Assume $S_s\in H^0(M,K^{-l}_{M})$, $x\in M$, satisfying $||S_s||^2_{H^{\otimes l}_{\omega_s},\omega_s}=1$,
\begin{equation*}
||S_s||_{H^{\otimes l}_{\omega_s}}^2(x)=
\sup_{||S||^2_{H^{\otimes l}_{\omega_s},\omega_s}=1}||S||^2_{H^{\otimes l}_{\omega_s}}(x)=\eta_{\omega_s}(x).
\end{equation*}
Then, by Lemma \ref{mhe} and Remark \ref{rhoremark}
\begin{equation*}
\begin{split}
\eta_{\omega_s}(x)&=\frac{||S_s||_{H^{\otimes l}_{\omega_s}}^2(x)}{||S_s||_{H^{\otimes l}_{\omega}}^2(x)}\cdot||S_s||^2_{H^{\otimes l}_{\omega},\omega}
\cdot\frac{||S_s||_{H^{\otimes l}_{\omega}}^2(x)}{||S_s||^2_{H^{\otimes l}_{\omega},\omega}}\\
&\le \frac{||S_s||_{H^{\otimes l}_{\omega_s}}^2(x)}{||S_s||_{H^{\otimes l}_{\omega}}^2(x)}\cdot||S_s||^2_{H^{\otimes l}_{\omega},\omega}\cdot \eta_{\omega}(x)\\
&\le C||S_s||^2_{H^{\otimes l}_{\omega},\omega}\cdot \rho_{\omega,l}(x).
\end{split}
\end{equation*}
Using Remark \ref{rhoremark} again, we obtain
\begin{equation*}
\rho_{\omega_s,l}(x)\le NC||S_s||^2_{H^{\otimes l}_{\omega},\omega}\cdot\rho_{\omega,l}(x).
\end{equation*}
Due to Zhang \cite{Z} or Ye\cite{Y2}, and the upper bound estimate of scalar curvature in Theorem \ref{main}, we have the $L^2$-Sobolev inequality holds, for all $\frac{1}{2}\le s \le1,$
\begin{equation*}
\left(\int_{M}\phi(x)^{\frac{2m}{m-1}}\omega^m_s\right)^{\frac{m-1}{m}}\le C\left(\int_M \left[4|\nabla \phi(x)|^2+\phi^2(x)\right]\omega^m_s\right), ~\forall \phi \in W^{1,2}(M,\omega_s).
\end{equation*}
Moreover, for section $S_s$, we have equation (see Tian \cite{T1})
\begin{equation*}
\Delta_s ||S_i||^2_{H^{\otimes l}_{\omega_s}}=||\nabla S_s||^2_{H^{\otimes l}_{\omega_s}}-ml||S_i||^2_{H^{\otimes l}_{\omega_s}}\ge -ml||S_i||^2_{H^{\otimes l}_{\omega_s}}.
\end{equation*}
Applying Moser's iteration, for all $y\in M$, we deduce
\begin{equation*}
||S_s||^2_{H^{\otimes l}_{\omega_s}}(y)\le C ||S_s||^2_{H^{\otimes l}_{\omega_s},\omega_s}\,=\,C.
\end{equation*}
By the equivalence of $H^{\otimes l}_{\omega_s}$ and $H^{\otimes l}_{\omega}$, we have
\begin{equation*}
||S_s||^2_{H^{\otimes l}_{\omega}}(y)\le C.
\end{equation*}
This implies
\begin{equation*}
||S_s||^2_{H^{\otimes l}_{\omega},\omega}\le CV.
\end{equation*}
Hence
\begin{equation*}
\rho_{\omega_s,l}(x)\le NC\rho_{\omega,l}(x).
\end{equation*}
The proof of the other part is similar.
\end{proof}
In order to prove Theorem \ref{thm23}, we only need to prove the following:
\begin{theorem}\label{bergman}
For any family of Fano manifolds $\{(M,\omega^i)\}$ with complex dimension $m=2,~3$ and $\Ric(\omega^i)\ge R_0,~\mbox{\Vol}(M,\omega^i)\ge V_0, ~\diam(M,\omega^i)\le d$, there exists a subsequence $\{(M,\omega^{i_k})\}$ and sequence $l_k\to \infty$, such that for all $l=l_k$
\begin{equation*}
\inf_{i_k}\inf_{x\in M}\rho_{\omega^{i_k},l}(x)>0.
\end{equation*}
\end{theorem}
\begin{proof}
Let $(M,\omega_s^i)$ be the manifolds at time $t=s$ along the K\"{a}hler Ricci flow with initial metric $\omega^i$. By the above theorem \ref{Beq} (the equivalence of Bergman kernels), if we can show a uniform lower bound for the Bergman kernel $\rho_{\omega_1^i,l}$ of $(M,\omega_1^i)$ with a sequence of $l\to \infty$, then the Theorem follows.  We need the following lemmas developed in Tian and Zhang's paper \cite{TZha1}.

First of all, the $L^4$-estimate of Ricci curvature is the key step to prove the theorem.
\begin{lemma}(\cite{TZha1})\label{L4estm} Let $(M,\omega)$ be a Fano manifold with complex dimension $m$, $\Ric(\omega)\ge R_0,~\mbox{\Vol}(M,\omega)\ge V_0>0, ~\diam(M,\omega)\le d$. Then along the K\"{a}hler Ricci flow
 \begin{equation}
\partial_t g_{i\bar{j}}\,=\,g_{i\bar{j}}\,-\,R_{i\bar{j}}=u_{i\bar{j}},~~\omega(0)=\omega,
\end{equation}
there exists $C=C(R_0,V_0,d,m)$ such that
\begin{equation}
\int_{M}|\Ric (\omega_t)|^4\omega^m_t\le C, ~~\forall t\in [\frac{1}{2},1].
\end{equation}
\end{lemma}
\begin{proof}
The proof is similar to \cite{TZha1} by noticing our estimates in Theorem \ref{main},
\begin{equation}\label{potest}
||u(t)||_{C^0}+||\nabla u(t)||_{C^0}+||\Delta u(t)||_{C^0}\le C,~~\forall t\in [\frac{1}{2},1]
\end{equation}
where $C=C(R_0,V_0,d,m)$. Combining (\ref{potest}) with the Chern-Weil theory, we have the following $L^2$-estimate
\begin{equation}
\int_M(|\nabla \nabla u|^2+|\nabla \overline{\nabla}u|^2+|Rm|^2)\omega^m_t\le C, ~\forall t\in [\frac{1}{2},1].
\end{equation}
Then by Bochner formula and integration by parts, we have $L^4$-estimate
\begin{equation}
\int_M|\nabla \overline{\nabla}u|^4\omega^m_t\le C\int_M(|\nabla\nabla \overline{\nabla}u|^2+|\overline{\nabla}\nabla\nabla u|^2)\omega^m_t,
\end{equation}
and
\begin{equation}
\int_M|\nabla {\nabla}u|^4\omega^m_t\le C\int_M(|\nabla\nabla {\nabla}u|^2+|\nabla\nabla \overline{\nabla}u|^2+|\overline{\nabla}\nabla\nabla u|^2)\omega^m_t
\end{equation}
for all $t\in [\frac{1}{2},1]$. In order to prove the $L^4$-bound of Ricci curvature, it suffices to estimate the $L^2$-bound of the third derivatives of $u$. Actually, by integration by parts and (\ref{potest}), we have
\begin{equation}\label{eqn16}
\begin{split}
\int_M(|\nabla\nabla {\nabla}u|^2&+|\nabla\nabla \overline{\nabla}u|^2+|\overline{\nabla}\nabla\nabla u|^2)\omega^m_t\\
&\le C\int_{M}(|\nabla \Delta u|^2+|Rm|^2+|\nabla \nabla u|^2)\omega^m_t, ~~\forall t\in [\frac{1}{2},1].
\end{split}
\end{equation}
In \cite{TZha1}, by using the evolution equation of $|\nabla \Delta u|^2$ and (\ref{eqn16}) the authors can estimate the upper bound of $\int_{M}|\nabla \Delta u|^2\omega^m_t$. Thus we finish the proof where the constants $C$ above depending only on $R_0,~V_0,~d$ and $m$. One can find more details in \cite{TZha1}.
\end{proof}
By using the methods developed by Cheeger-Colding \cite{ChCo,ChCo97,ChCo00} and Cheeger-Colding-Tian \cite{ChCoTi02} and Petersen-Wei \cite{PeWe,PeWe1}, the authors in \cite{TZha1} can prove
\begin{lemma}(\cite{TZha1})
Let $\{(M_i,\omega^i)\}$ be a family of Fano manifolds  with complex dimension $m$ and
\begin{equation}
\int_{M_i}|\Ric|^p\le \Lambda.
\end{equation}
We further assume the non-collapsing, $\Vol(B_r(x))\ge \kappa r^{2m}$, for all $x\in M_i,~ r\le 1$, where $p>{m}$ and $\kappa,~\Gamma$ are uniform positive constants. Then there exists a subsequence such that $(M_i,\omega^i)$ is convergent in the pointed Gromov-Hausdorff topology
\begin{equation}
(M_i,\omega^i)\xrightarrow[]{d_{GH}} (M_{\infty},d),
\end{equation}
and the followings hold,\\
(i) $M_{\infty}=\mathcal {S}\cup \mathcal{R} $ such that the singular set $\mathcal{S}$ is a closed set of codimension $\ge 4$
and  $\mathcal{R}$ is convex in $M_{\infty}$;\\
(ii) There exists a $C^{\alpha},~ \forall \alpha <2-\frac{m}{p}$, metric $g_{\infty}$ on $\mathcal{R}$ which induces $d$; \\
(iii) $\omega^i$ converges to $g_\infty$ in the $C^{\alpha}$ topology on $\mathcal{R}.$\qed
\end{lemma}
Combining the $L^4$-bound estimate of Ricci curvature along the K\"{a}hler Ricci flow (Lemma \ref{L4estm}) and applying Perelman's pseudolocality theorem \cite{Pe} of Ricci flow and Shi's higher derivative estimate to curvature \cite{Shi}, with the same argument in section 3.3 of \cite{TZha1} we can show that the convergence is smooth on the regular set along the K\"{a}hler Ricci flow:
\begin{lemma}(\cite{TZha1})
With the same assumptions as Theorem \ref{bergman}, denote by $\omega_1^i$, the $1$ time slice of the K\"ahler Ricci flow starting from $\omega^i$. Then up to a subsequence we have
\begin{equation}(M,\omega^i_1)\xrightarrow[]{d_{GH}} (M_{\infty},d),
 \end{equation}
 and the limit $M_\infty$ is smooth outside a closed subset $\mathcal{S}$ of real codimension $\ge 4$ and $d$ is induced by a smooth K\"{a}hler metric $g_\infty$ on $M_{\infty}\backslash \mathcal{S}$. Moreover, $\omega_1^i$ converge to $g_\infty$ in $C^\infty$-topology outside $\mathcal{S}.$\qed
\end{lemma}
Now by noticing the estimate of Ricci potential,
\begin{equation}
||u(t)||_{C^0}+||\nabla u(t)||_{C^0}+||\Delta u(t)||_{C^0}\le C,~~\forall t\in [\frac{1}{2},1]
\end{equation}
and the smooth convergence of $(M,\omega^i_1)$, we can finish the proof of Theorem \ref{bergman} with the similar arguments as in \cite{T3,T5}, see also the proof of Theorem 5.1 in \cite{TZha1}. We remark that the
$C^1$ estimate to $u(t)$ will be used in the iteration arguments to cancel the bad
terms containing $\nabla \overline{\nabla} u(t)$.
\end{proof}

Since the nonnegative holomorphic bisectional curvature is preserved along the K\"{a}hler Ricci flow(see \cite{Mo}), then combining with the upper bound estimate of scalar curvature (Theorem \ref{main}), the diameter estimate (Lemma \ref{diame}) and Hamilton's compactness theorem(\cite{Ha}), we have
\begin{theorem}
Let $(M,\omega)$ be a Fano manifold of complex dimension $m\ge 1$, with nonnegative holomorphic bisectional curvature, volume lower bound $V_0>0$ and diameter upper bound $d$. Then for all $l\in \mathbb{N}_{+}$, the Bergman kernel $\rho_{\omega,l}$ has a uniform positive lower bound
\begin{equation}
\rho_{\omega,l}>c_l>0
\end{equation}
where the constant $c_{l}$ depends only on $m,~V_0,~d$ and $l$.\qed
\end{theorem}

\section*{Acknowledgments}
The author would like to thank his advisor Gang Tian for suggesting this problem and constant encouragement and several useful comments on an earlier version of this paper. I also like to thank Professor Zhenlei Zhang for helpful conversations about his joint paper with Tian. I also like to thank Feng Wang who
taught me so much about the partial $C^0$-estimate and for many helpful conversations. I also like to thank Huabin Ge for many helpful suggestions
on this topic and so many useful conversations.

\bibliographystyle{amsplain}

\end{document}